\documentclass[11pt]{article}
\usepackage[english]{babel}
\usepackage{times}
\usepackage{paralist}
\usepackage{amsmath,amstext,amssymb,amsthm, amsfonts}
\usepackage[dvips]{graphics,color}
\newcommand{\bi}{\begin{itemize}}  
\newcommand{\ei}{\end{itemize}}     
\newcommand{\bc}{\begin{center}}  
\newcommand{\ec}{\end{center}}     

\newcommand{\ls}[1]
   {\dimen0=\fontdimen6\the\font \lineskip=#1\dimen0
   \advance\lineskip.5\fontdimen5\the\font \advance\lineskip-\dimen0
   \lineskiplimit=.9\lineskip \baselineskip=\lineskip
   \advance\baselineskip\dimen0 \normallineskip\lineskip
   \normallineskiplimit\lineskiplimit \normalbaselineskip\baselineskip
   \ignorespaces }

\numberwithin{equation}{section}

\newcommand{\ilim} {\mathop{\rm lim\,inf}}

\newtheorem{lemma}{Lemma}[section]
\newtheorem{theorem}[lemma]{Theorem}
\newtheorem{corollary}[lemma]{Corollary}

\newtheorem{example}[lemma]{Example}

\newtheorem{remark}[lemma]{Remark}

\def\S{\mathbb{S}}
\def\K{\mathbb{K}}
\def\Y{\mathbb{Y}}
\def\h{\mathbf{I}}
\def\X{\mathbb{X}}
\def\E{\mathbb{E}}
\def\A{\mathbb{A}}
\def\H{\mathbb{H}}
\def\R{\mathbb{R}}
\def\P{\mathbb{P}}
\def\F{\mathbb{F}}

\def\C{\mathbb{C}}
\def\W{\mathbb{W}}
\def\c{\bar{c}}
\def\B{\mathcal{B}}
\def\oo{\mathcal{O}}

\def\dist{{\rho_{tv}}}
\newcommand\Wc{{\mathop{\longrightarrow}\limits^w}}
\newcommand\Sc{{\mathop{\longrightarrow}\limits^s}}

\def\I{\mathbf{I}}

\title{Convergence of Probability Measures and Markov Decision Models with Incomplete Information}

\begin{document}

\maketitle

\begin{center}
Eugene~A.~Feinberg \footnote{Department of Applied Mathematics and
Statistics,
 Stony Brook University,
Stony Brook, NY 11794-3600, USA, eugene.feinberg@sunysb.edu},\
Pavlo~O.~Kasyanov\footnote{Institute for Applied System Analysis,
National Technical University of Ukraine ``Kyiv Polytechnic
Institute'', Peremogy ave., 37, build, 35, 03056, Kyiv, Ukraine,\
kasyanov@i.ua.},\ and Michael~Z.~Zgurovsky\footnote{National Technical University of Ukraine
``Kyiv Polytechnic Institute'', Peremogy ave., 37, build, 1, 03056,
Kyiv, Ukraine,\
 zgurovsm@hotmail.com
}\\

\bigskip
\end{center}

\centerline{\emph{This article is dedicated to 80th birthday of Academician Albert Nikolaevich Shiryaev}}

\begin{abstract}
This paper deals with three major types of convergence of
probability measures on metric spaces: weak convergence, setwise
converges, and convergence in the total variation.  First, it
describes and compares necessary and sufficient conditions
for these types of convergence, some of which are well-known, in
terms of convergence of probabilities of open and closed sets and,
for the probabilities on the real line, in terms of convergence of
distribution functions.  Second, it provides 
criteria for weak and setwise convergence of probability measures and continuity of stochastic  kernels in terms of convergence of
probabilities defined on the base of the topology generated by the
metric.  Third, it provides applications to control of Partially Observable Markov Decision Processes and, in particular, to Markov
Decision Models with incomplete information.
\end{abstract}

\section{Introduction} \label{S1}
This paper deals with convergence of probability measures and relevant
applications to control of stochastic systems with incomplete state
observations.  Convergence of probability measures and control of stochastic systems
under incomplete information are among the areas to which Albert Nikolayevich Shiryaev has made
fundamental contributions.  In particular, convergence of
probability measures and limit theorems for stochastic processes
were studied in his joint papers with his distinguished students
Yuri Mikhailovich
Kabanov and Robert Shevilevich
Liptser (e.g.,
\cite{KLSh}) and in his monograph with Jean Jacod~\cite{JSh}.  Control
of stochastic processes with incomplete information was the major
topic of his two influential papers \cite{Sh1, Sh2}, and this
topic is related to his monograph with Liptser \cite{LSh} on
statistics of stochastic processes.

In Section~\ref{S2} of this paper we describe  three major types
of convergence of probability measures defined on metric spaces:
weak convergence, setwise convergence, and convergence in the
total variation.  In addition to the definitions, we provide two
groups of mostly known results: characterizations of these types of convergence
via convergence of probability measures of open and closed sets, and, for probabilities on a real line, 
via convergence of
distribution functions. In section~\ref{S3} we describe 
criteria for weak and setwise convergences in terms of convergence of
probabilities of the elements of a countable base of the topology. Section~\ref{3A} deals with continuity of transition probabilities.
In particular, Theorem~\ref{mainthkern} describes sufficient conditions for a probability measure, defined on a product of two spaces and depending on a parameter, to have a transition probability satisfying certain continuity properties.  This result can be interpreted as a sufficient condition for  continuity  in Bayes's formula.
  Section~\ref{S4} describes
recent results on optimization of Partially Observable Markov
Decision Processes (POMDPs) from Feinberg et al.~\cite{FKZ} as well as new results.
Section~\ref{S5} describes an application of the results from
Sections~\ref{3A} and \ref{S4} to a particular class of POMDPs, that we call
Markov Decision Models with Incomplete Information ({MDMIIs}). The
difference between a POMDP and an MDMII is that for a POMDP the
states of the system and observations are related via a stochastic kernel, called an observation stochastic kernel,
while for an MDMII the
state of the system is a vector, consisting of $(m+n)$
coordinates, of which $m$ coordinates are observable and $n$
coordinates are not observable.  MDMIIs were studied mainly in
early publications including in Aoki~\cite{Ao}, Dynkin~\cite{Dy}, Shiryaev~\cite{Sh2},
Hinderer~\cite{Hi}, Savarigi and Yoshikava~\cite{SY},
Rhenius~\cite{Rh}, Rieder~\cite{Ri}, Yushkevich~\cite{Yu}, Dynkin
and Yushkevich~\cite{DY}, and B\"auerle and Rieder~\cite{BR}, while
POMDPs were studied by Bertsekas and Shreve~\cite{BS},
Hern\'andez-Lerma~\cite{HL}, and in many later publications.

Feinberg et al.~\cite{FKZ} described sufficient conditions for the
existence of optimal policies, validity of optimality equations,
and convergence of value iterations to optimal values for POMDPs
 with standard Borel state, action, and observation spaces and for MDMIIs
 with standard Borel state and action spaces; see also conference and seminar proceedings \cite{FKZ1, FKZ2}. In both cases, the goal is
 either to minimize the expected total costs, with the one-step cost
 function being nonnegative, or to minimize the expected total discounted cost,
with the one-step cost function being bounded below.
  For POMDPs these sufficient conditions
 are: $K$-inf-compactness of the cost function, weak continuity of
 the transition stochastic kernel, and continuity in the total variation of
 the observation stochastic kernel.  These results are described in Section~\ref{S4} as well as sufficient conditions for weak continuity of transition probabilities for a COMDP from Feinberg et al.~\cite{FKZ} in terms of the transition function $H$ in the filtering equation~(\ref{3.1}). In this paper we   introduce sufficient conditions in terms of joint distributions of posteriory distributions and observations; see Theorem~\ref{teor:Rtotvar}.  The notion of $K$-inf-compactness of a function defined on a graph of
 a set-valued map was introduced in Feinberg et al.~\cite{FKN}.

 Though an MDMII is a particular case of
 an
 POMDP, there is no observation stochastic kernel  in the definition of an
 MDMII.  However, the observation stochastic kernel  can be defined for an MDMII in a natural way, and this definition
 transforms an MDMII into a POMDP, but in this 
 POMDP the defined observation stochastic kernel  is not continuous in the
 total variation.  Feinberg et al.~\cite{FKZ} described additional equicontinuity
 conditions on the stochastic  kernels of MDMIIs, under which  optimal
 policies exist, optimality equations hold, and value iterations
converge to optimal values.  By using results from Sections~\ref{3A} and \ref{S4}, in Section~\ref{S5}
 we strengthen the results from Feinberg et al.~\cite{FKZ} on MDMIIs by providing weaker assumptions
 on transition probabilities than the assumptions introduced in Feinberg et al.~\cite{FKZ}.

\section{Three types of convergence of probability
measures}\label{S2} Let $\S$ be a metric space and  ${\mathcal
B}(\S)$ be its Borel $\sigma$-field, that is, the $\sigma$-field
generated by all open subsets of the metric space $\S$.  For
$S\in\B( \S)$  denote by ${\mathcal B}(S)$ the $\sigma$-field
whose elements are intersections of $S$ with elements of
${\mathcal B}(\S)$. Observe that $S$ is a metric space with the
same metric as on $\S$, and ${\mathcal B}(S)$ is its Borel
$\sigma$-field. For a metric space $\S$,  denote by $\P(\S)$ the
\textit{set of probability measures} on $(\S,{\mathcal B}(\S)).$ A
sequence of probability measures $\{P_n\}_{n=1,2,\ldots}$ from
$\P(\S)$ \textit{converges weakly (setwise)} to $P\in\P(\S)$ if
for any bounded continuous (bounded Borel-measurable) function $f$
on $\S$
\[\int_\S f(s)P_n(ds)\to \int_\S f(s)P(ds) \qquad {\rm as \quad
}n\to\infty.
\]
We write $P_n\Wc P$ ($P_n\Sc P$) if the sequence
$\{P_n\}_{n=1,2,\ldots}$ from $\P(\S)$ converges weakly (setwise)
to $P\in\P(\S).$  The definition of Lebesgue-Stiltjes integrals
implies that $P_n\Sc P$ if and only if $P_n(E)\to P(E)$ for each
$E\in{\cal B}(\S)$ as $n\to\infty.$
%
The following two theorems are well-known.

\begin{theorem}\label{t1} {\rm (Shiryaev~\cite[Theorem 1, p.
311]{Sh}).}
The following statements are equivalent:

(i) $P_n\Wc P;$

(ii) $\liminf_{n\to\infty} P_n(\oo)\ge P(\oo)$ for each open
subset $\oo\subseteq\S;$

(iii) $\limsup_{n\to\infty} P_n(C)\le P(C)$ for each closed subset
$C\subseteq \S.$
\end{theorem}

Let $\R^1$ be a real line with the Euclidean metric.  For a
$P,P_n\in \P(\R^1)$ define the distribution functions
$F(x)=P\{(-\infty,x] \}$ and $F_n(x)=P_n\{(-\infty,x] \},$
$x\in\R^1.$
\begin{theorem}\label{t2} {\rm (Shiryaev~\cite[Theorem 2, p.
314]{Sh}).} 
For $\S=\R^1$ the following statements are equivalent:

(i) $P_n\Wc P;$

(ii) $F_n(x)\to F(x)$ for all points $x\in\R^1$ of continuity of
the distribution function $F$.
\end{theorem}

The following theorem provides for setwise convergence the
results in the same spirit as Theorem~\ref{t1} 
states for weak convergence.

\begin{theorem}\label{t3}
The following statements are equivalent:

(i) $P_n\Sc P;$

(ii) $\lim_{n\to\infty} P_n(\oo)= P(\oo)$ for each open subset
$\oo\subseteq\S;$

(iii) $\lim_{n\to\infty} P_n(C)= P(C)$ for each closed subset
$C\subseteq \S.$
\end{theorem}
\begin{proof} If $A$ is open (closed) then its complement $A^c$ is
closed (open), and $Q(A^c)=1-Q(A)$ for each $Q\in\P(\S).$  Thus
statements (ii) and (iii) are equivalent. We prove the equivalence
of (i) and (iii).  Obviously, (i) implies (iii). According to
Billingsley \cite[Theorem~1.1]{Bil} or Bogachev \cite[Theorem
7.1.7]{bogachev}, any probability measure $P$ on a metric space
$\S$ is regular, that is, for each $B\in \B(\S)$ and for each
$\varepsilon>0$ there exist a closed subset $C\subseteq\S$ and an
open subset $\oo\subseteq \S$ such that $C\subseteq B\subseteq
\oo$ and $P(\oo\setminus C)<\varepsilon$.
 Fix arbitrary $B\in \B(\S)$
and $\varepsilon>0$.  Since $P_n(\oo)\to P(\oo)$ and $P_{n}(C)\to
P (C)$, there exists $N=1,2,\ldots,$ such that $|P_n(\oo)-
P(\oo)|<\varepsilon$ and $|P_n(C)-P (C)|<\varepsilon$ for any $n=
N,N+1,\ldots$. Therefore, $P_n(B)-P(B)\le P_n(\oo)- P(B)<
\varepsilon + P(\oo\setminus C)<2\varepsilon$, and $P(B)-P_n(B)\le
P(B)-P_n(C)< \varepsilon+P(\oo\setminus C)<2\varepsilon$, for each
$n=N,N+1,\ldots$. Since $\varepsilon>0$ is arbitrary, the sequence
$\{P_n(B)\}_{n=1,2,\ldots} \subset [0,1]$ converges to $P(B)$ for
any $B\in \B(\S)$, that is, the sequence of probability measures
$\{P_n\}_{n=1,2,\ldots}$ converges setwise to $P\in\P(\S)$.
\end{proof}

According to Bogachev \cite[Theorem 8.10.56]{bogachev}, which is
Pflanzagl's generalization of the
Fichtengolz-Dieudonn\'e-Grothendiek theorem, the statement of
Theorem~\ref{t3} holds for Radon measures. In view of Bogachev
\cite[Theorem 7.1.7]{bogachev}, if $\S$ is complete and separable,
then any probability measure on $(\S,{\mathcal B}(\S))$ is Radon.
However, Theorem~\ref{t3} does not assume that $\S$ is either
separable or complete.

 If $P_n\Sc P$, where $P, P_n\in\P(\R^1)$ for all $n=1,2,\ldots,$ then $F_n(x) \to
 F(x)$ and $F_n(x-)\to F(x-)$ for all $x\in\R^1.$  This is true because
$F_n(x) = P_n((-\infty,x])\to P((-\infty,x])= F(x)$ and $F_n(x-) =
P_n((-\infty,x))\to P((-\infty,x))= F(x-)$ as $n\to\infty.$
However, as the following example shows, the convergences  $F_n(x)
\to
 F(x)$ and $F_n(x-)\to F(x-)$ for all $x\in\R^1$ do not imply $P_n\Sc P.$

\begin{example}\label{exa:DFS}(Convergences $F_n(x)
\to
 F(x)$ and $F_n(x-)\to F(x-)$  $\forall x\in\R^1$ do not imply $P_n\Sc P$).
{\rm
Let
\[
F_0(x):=\left\{
\begin{array}{ll}
0,& x<0;\\
x, & 0\le x \le 1;\\
1, & x>1;
\end{array}
\right.
F_{n+1}(x):=\left\{
\begin{array}{ll}
\frac12F_n(3x),& x<\frac13;\\
\frac12, & \frac13 \le x\le \frac23;\\
\frac12F_n(3x-2), & x>\frac23;
\end{array}
\right.
F(x):=\left\{
\begin{array}{ll}
0,& x<0;\\
C(x), & 0 \le x\le 1;\\
1, & x>1;
\end{array}
\right.
\]
where $C(x)$ is the Cantor function and $n=0,1,\ldots\ .$ Note that $F(x)$ and $F_n(x)$, $n=0,1,\ldots,$ are continuous functions and
\[
\max_{x\in\R^1}\left|F(x)-F_n(x) \right|\le 2^{1-n}\max_{x\in\R^1}\left|F_1(x)-F_0(x) \right|,\quad n=1,2,\ldots\ .
\]
Therefore, $F_n(x-)=F_n(x)\to
 F(x)= F(x-)$ for each $ x\in\R^1$. 

Denote by $C\subset [0,1]$ the Cantor set. Since the Lebesgue measure of the Cantor set $C$ equals  zero and each distribution function $F_n$ has a bounded density,  $P_n(C)=0$ for each $n=1,2,\ldots.$ Note that $P(C)=1$ because $P([0,1])=F(1)-F(0)=1$ and $P([0,1]\setminus C)=0$ since $[0,1]\setminus C$ is a union of disjoint open interval each of zero $P$-measure.  Thus, the sequence of probability measures $\{P_n\}_{n=1,2,\ldots}$ does not converges setwise to the probability measure $P$. \hfill$\Box$
}
\end{example}

%
%
%
%
%
%
%
%
%
%
%
%

The third major type of convergence of probability measures,
convergence in the total variation, can be defined via a metric
$\rho_{tv}$  on $\P(\S)$ called the distance in the total variation.  For $P,Q\in \P(\S),$ define
\begin{equation}\label{eqdefdist}
\dist(P,Q):=\sup\left\{|\int_\S f(s)P(ds)-\int_\S f(s)Q(ds)| : \
f:\S\to [-1,1]\mbox{ is Borel-measurable} \right\}. \end{equation}
A sequence of probability measures $\{P_n\}_{n=1,2,\ldots}$ from
$\P(\S)$ converges in the total variation to $P\in\P(\S)$ if
$\lim_{n\to\infty}\dist(P_n,P)=0.$

In view of the Hahn decomposition, there exists $E\in{\cal B}(\S)$
such that $(P-Q)(B)\ge 0$ for each $B\in {\cal B}(E)$ and
$(P-Q)(B)\le 0$ for each $B\in {\cal B}(E^c).$  According to
Shiryaev~\cite[p. 360]{Sh},
\begin{equation}\label{eq2shir}\dist(P,Q)=P(E)-Q(E)+Q(E^c)-P(E^c)=2\sup \{|P(B)-Q(B)|:B\in{\cal B}(\S)\}.\end{equation}
This implies that the supremum in (\ref{eqdefdist}) is achieved at
the function $f(s)={\bf I}\{s\in E\}-{\bf I}\{s\in E^c\},$ and
\begin{equation}\label{eqdefdist1}\dist(P,Q)=\sup\left\{\int_\S f(s)P(ds)-\int_\S f(s)Q(ds) : \
f:\S\to \{-1,1\}\mbox{ is Borel-measurable}
\right\}.\end{equation} Since $(P-Q)(\S)=0$, (\ref{eq2shir}) also
implies
\begin{equation}\label{eq2shir1}
\dist(P,Q)=2P(E)-2Q(E)=2Q(E^c)-2P(E^c)= 2\max \{P(B)-Q(B):B\in{\cal B}(\S)\}.
\end{equation}

Consider the positive part $(P-Q)^+$ and negative part $(P-Q)^-$
of $(P-Q)$, that is, $(P-Q)^+(B) = (P-Q)(E\cap B)$ and
$(P-Q)^-(B)=-(P-Q)(E^c\cap B)$ for all 
$B\in\B(\S)$. Both $(P-Q)^+$ and $(P-Q)^-$ are nonnegative finite
measures. As follows from (\ref{eq2shir1}),
\begin{equation}\label{eq:shir3}
\dist(P,Q)=2(P-Q)^+(E)=2(P-Q)^-(E^c).
\end{equation}

The statements of Theorem~\ref{t5}(i,ii) characterize convergence in the total
variation via convergence of the values of the measures on open
and closed subsets in $\S.$  In this respect, these statements are
similar to Theorems~\ref{t1} and \ref{t3}, which provide
characterizations for weak and setwise convergences.
Formula~(\ref{eq2shir}) indicates that convergence in the total
variation can be interpreted as uniform setwise convergence. The same interpretation follows from Theorems~\ref{t3} and \ref{t5}(i, ii).
Theorem~\ref{t5}(iii, iv) indicates that convergence in the total
variation can be also interpreted as uniform weak convergence.

\begin{theorem}\label{t5}  The following equalities hold for $P,Q\in \P(\S)$:

(i) $\dist(P,Q)=2\sup\{|P(C)-Q(C)|:C\  {\rm is\  closed\ in}\
\S\}=2\sup\{P(C)-Q(C):C\  {\rm is\  closed\ in}\ \S\};$

(ii) $\dist(P,Q)=2\sup\{|P(\oo)-Q(\oo)|:\oo\  {\rm is\  open\ in}\
\S\}=2\sup\{P(\oo)-Q(\oo):\oo\  {\rm is\  open\ in}\ \S\};$

(iii) $\dist(P,Q)=\sup\left\{\int_\S f(s)P(ds)-\int_\S f(s)Q(ds) :
\ f:\S\to [-1,1]{\rm\ is\ continuous} \right\};$

(vi) $\dist(P,Q)=\sup\left\{|\int_\S f(s)P(ds)-\int_\S f(s)Q(ds)|
: \ f:\S\to [-1,1]{\rm\ is\ continuous} \right\}.$
\end{theorem}
\begin{proof}
(i) It is sufficient to show that
\begin{equation}\label{eq:tv(i,ii)}
\dist(P,Q)\le 2\sup\{P(C)-Q(C):C\  {\rm is\  closed\ in}\ \S\}.
\end{equation}  Since $(P-Q)^+$ is a measure   on a
metric space, it is regular; Billingsley \cite[Theorem~1.1]{Bil}
or Bogachev \cite[Theorem 7.1.7]{bogachev}. Thus, for $E\in
\B(\S)$ satisfying (\ref{eq:shir3}) and for each $\varepsilon>0$ there
exists a closed subset $C\subseteq \S$ such that $C\subseteq E$
and $2(P-Q)^+(E\setminus C)<\varepsilon$. 
Due to $C\subseteq E,$ the equality $(P-Q)(C)=(P-Q)^+(C)$ holds.
 Therefore, in view of
(\ref{eq:shir3}),
\[
\dist(P,Q)< 2(P-Q)^+(C)+\varepsilon\le 2\sup\{P(C)-Q(C):C\ {\rm
is\  closed\ in}\ \S\}+\varepsilon.
\]
Since $\varepsilon>0$ is an arbitrary, inequality (\ref{eq:tv(i,ii)}) holds. 

(ii) Since of $\dist(P,Q)=\dist(Q,P)$ and
\[
\sup\{P(C)-Q(C):C\  {\rm is\  closed\ in}\ \S\}=\sup\{Q(\oo)-P(\oo):\oo\  {\rm is\  open\ in}\ \S\},
\]
(i) implies  (ii).

(iii) In view of (\ref{eqdefdist1}), it is sufficient to show that
\begin{equation}\label{eq:tv(iii)}
\dist(P,Q)\le \sup\left\{\int_\S f(s)P(ds)-\int_\S f(s)Q(ds)
: \ f:\S\to [-1,1]{\rm\ is\ continuous} \right\}.
\end{equation}

Since the supremum in (\ref{eqdefdist}) is achieved at the
function $f_{E,E^c}(s)={\bf I}\{s\in E\}-{\bf I}\{s\in E^c\},$
\begin{equation}\label{eq:tv(iii)(1)}
\dist(P,Q)=\int_\S f_{E,E^c}(s)(P-Q)(ds).
\end{equation}

Since of $(P-Q)^+$ and $(P-Q)^-$ are measures on a metric space,
they are regular; Billingsley \cite[Theorem~1.1]{Bil} or Bogachev
\cite[Theorem 7.1.7]{bogachev}. Thus, for $E, E^c\in \B(\S)$ and
for each $\varepsilon>0,$ there exist closed subsets
$C_1,C_2\subseteq \S$ such that $C_1\subseteq E$, $C_2\subseteq
E^c$, and $(P-Q)^+(E\setminus C_1)+(P-Q)^-(E^c\setminus
C_2)<\varepsilon$. Therefore,
\begin{equation}\label{eq:tv(iii)(2)}
\int_\S f_{E,E^c}(s)(P-Q)(ds)\le \int_\S f_{C_1,C_2}(s)(P-Q)(ds)+\varepsilon,
\end{equation}
where $f_{C_1,C_2}(s)={\bf I}\{s\in C_1\}-{\bf I}\{s\in C_2\}$,
$s\in \S$. Note that the restriction of $f_{C_1,C_2}$ on a closed
subset $C_1\cup C_2$ in $\S$ is continuous. Since a metric space
is a normal topological space, Tietze-Urysohn-Brouwer extension
theorem implies the existence of a continuous extension of
$f_{C_1,C_2}$ on $\S$, that is, there is a continuous function
$\tilde{f}_{C_1,C_2}:\S\to [-1,1]$ such that
$\tilde{f}_{C_1,C_2}(s)=f_{C_1,C_2}(s)$ for any $s\in C_1\cup
C_2$. Thus,
\begin{equation}\label{eq:tv(iii)(3)}
\int_\S f_{C_1,C_2}(s)(P-Q)(ds)\le \int_\S \tilde{f}_{C_1,C_2}(s)(P-Q)(ds)+\varepsilon.
\end{equation}

According to 
(\ref{eq:tv(iii)(1)})--(\ref{eq:tv(iii)(3)}), for any
$\varepsilon>0$
\[
\dist(P,Q)\le \sup\left\{\int_\S f(s)P(ds)-\int_\S f(s)Q(ds)
: \ f:\S\to [-1,1]{\rm\ is\ continuous} \right\}+2\varepsilon,
\]
which yields inequality (\ref{eq:tv(iii)}).

(iv) According to (iii) and the definition of $\dist(P,Q)$,
\[
\begin{aligned}
&\dist(P,Q)=\sup\left\{\int_\S f(s)P(ds)-\int_\S f(s)Q(ds)
: \ f:\S\to [-1,1]{\rm\ is\ continuous} \right\}\le \\
&\sup\left\{|\int_\S f(s)P(ds)-\int_\S f(s)Q(ds)|
: \ f:\S\to [-1,1]{\rm\ is\ continuous} \right\}\le \dist(P,Q),
\end{aligned}
\]
which implies (iv).
\end{proof}

For a function $f$ on $\R$, let $V(f)$ denote its total variation.
Let $P_i,$ $i=1,2,$ be probability measures on $(\R^1,\B(\R^1)),$ and $F_i(x)=P_i\{(-\infty,x]\},$ $x\in\R^1,$  be the corresponding distribution functions. 
The following well-known statement characterizes
convergence in the total variation in terms of convergence of
distribution functions.

\begin{theorem}\label{t:totvar2} {\rm (Cohn \cite[Exercise~6, p.~137]{Cohn}).} $\dist(P_1,P_2)=V(F_1-F_2)$ for
all $P_1,P_2\in\P(\R^1).$
\end{theorem}

\section{Sufficient Conditions for Weak and Setwice Convergence}\label{S3}

\begin{lemma}\label{l:1} Let
$\{P_n\}_{n=1,2,\ldots}$ be a sequence of probability measures
from $\P(\S)$ and $P\in\P(\S)$.  If for a measurable subset $B$ of
$\S$ there is a countable sequence of measurable subsets
$B_1,B_2,\ldots$ of $B$ such that:

(i) $B=\cup_{i=1}^\infty B_j,$

(ii) $\liminf_{n\to\infty}P_n(\cup_{j=1}^k B_j)\ge P(\cup_{j=1}^k
B_j)$  for all $k=1,2,\ldots,$

 \noindent then
\begin{equation}\label{eq3.1nn}\liminf_{n\to\infty} P_n(B)\ge P(B).\end{equation}
\end{lemma}
\begin{proof}
For an arbitrary $\epsilon>0$ consider an integer $k(\epsilon)$
such that $P(\cup_{j=1}^{k(\epsilon)} B_j)\ge P(B)-\epsilon.$ Then
\[ \liminf_{n\to\infty}
P_n(B)\ge \liminf_{n\to\infty} P_n(\cup_{j=1}^{k(\epsilon)}
B_j)\ge P(\cup_{j=1}^{k(\epsilon)} B_j) \ge P(B)-\epsilon. \]
Since $\epsilon>0$ is arbitrary, inequality (\ref{eq3.1nn})
holds.\end{proof}

\begin{corollary}\label{c:1}
Let $\{P_n\}_{n=1,2,\ldots}$ be a sequence of probability measures
from $\P(\S)$ and $P\in\P(\S)$.  If for a each open subset $\oo$
of $\S$ there is a countable sequence of measurable subsets
$B_1,B_2,\ldots$ of $\oo$ such that:

(i) $\oo=\cup_{i=1}^\infty B_j,$

(ii) $\liminf_{n\to\infty}P_n(\cup_{j=1}^k B_j)\ge P(\cup_{j=1}^k
B_j)$  for all $k=1,2,\ldots,$

 \noindent then \noindent then $P_n\Wc P$.
\end{corollary}
\begin{proof}
In view of Lemma~\ref{l:1}, $\liminf_{n\to\infty} P_n(\oo)\ge
P(\oo)$ for all open subsets $\oo$ of $\S.$  In view of Theorem~\ref{t1}, this is equivalent to
$P_n\Wc P$.
\end{proof}

\begin{theorem}\label{t:1} Let
$\{P_n\}_{n=1,2,\ldots}$ be a sequence of probability measures
from $\P(\S)$ and $P\in\P(\S)$. If the topology on $\S$ has a
countable base $\tau_b,$ then $P_n\Wc P$ if and only if
$\liminf_{n\to\infty}P_n(\oo^*)\ge P(\oo^*)$ for each finite union
$\oo^*=\cup_{i=1}^k {\oo}_{i}$ with
$\oo_{i}\in\tau_b,$ $k=1,2,\ldots\ .$ 
%
\end{theorem}
\begin{proof}  Since $P_n\Wc P$ if an only if $\liminf_{n\to\infty}P_n(\oo)\ge P(\oo)$ for each
open $\oo\subseteq\S,$ the necessary condition is obvious.  The
sufficient part follows from Corollary~\ref{c:1}, because any open
subset $\oo$ of $\S$ can be represented as
$\oo^*=\cup_{i=1}^\infty {\oo}_{i}$ with $\oo_{i}\in\tau_b,$
$i=1,2,\ldots\ .$
\end{proof}

Lemma~\ref{l:1} can be used to formulate the following criterion
for setwise convergence.

\begin{lemma}\label{l:2} Let
$\{P_n\}_{n=1,2,\ldots}$ be a sequence of probability measures
from $\P(\S)$ and $P\in\P(\S)$.  Then the following statements
hold:

(i) If for a measurable subset $C$ of $\S,$   both sets $B=C$
and $B=C^c$, where $C^c=\S\setminus C$ is the complement of $C,$
satisfy the conditions of Lemma~\ref{l:1}, then $P_n(C)\to P(C).$

(ii) If for each open subset $\oo\subseteq\S,$    both sets
$B=\oo$ and its complement $B=\oo^c$ satisfy conditions (i) and (ii) of
Lemma~\ref{l:1}, then $P_n\Sc P.$
\end{lemma}
\begin{proof}
(i) Lemma~\ref{l:1} implies that $\liminf_{n\to\infty} P_n(C)\ge
P(C)$ and  $\liminf_{n\to\infty} P_n(C^c)\ge P(C^c).$  Since $P$
and $P_n,$ $n=1,2,\ldots$ are probability measures,
$\lim_{n\to\infty} P_n(C)= P(C).$  (ii) In view of (i),
$P_n(\oo)\to P(\oo)$ for each open subset $\oo$ of $\S.$  In view
of Theorem~\ref{t3}, $P_n\Sc P.$
\end{proof}

For setwise convergence the following theorem states the
conditions similar to the conditions of
Theorem~\ref{t:1} for weak convergence.

\begin{theorem}\label{t:2} Let
$\{P_n\}_{n=1,2,\ldots}$ be a sequence of probability measures
from $\P(\S)$ and $P\in\P(\S)$. If the topology on $\S$ has a
countable base $\tau_b,$ then $P_n\Sc P$ if and only if the
following two conditions hold:

(i) $\liminf_{n\to\infty}P_n(\oo^*)\ge P(\oo^*)$ for each finite
union $\oo^*=\cup_{i=1}^ k {\oo}_{i}$, where $\oo_{i}\in\tau_b,$
$k=1,2,\ldots;$

(ii) each closed subset $B\subseteq \S$ satisfies  conditions (i) and (ii)
of Lemma~\ref{l:1}.
%
%
\end{theorem}
\begin{proof}
Let $\oo$ be an arbitrary open subset of $\S.$  In view of (i),
Theorem~\ref{t1} implies that $\liminf_{n\to\infty} P_n(\oo)\ge P(\oo).$  In
view of (ii), Lemma~\ref{l:1}  implies that  $\liminf_{n\to\infty}
P_n(\oo^c)\ge P(\oo^c).$  Thus $\lim_{n\to\infty} P_n(\oo)= P(\oo).$  Since
$\oo$ is an arbitrary open subset of $\S,$  Theorem~\ref{t3}
implies that $P_n\Sc P.$
\end{proof}

In some applications, it is more convenient to verify convergence
of probabilities for intersections of events than for unions of
events.  The following lemma links the convergence of probabilities
for intersections and unions of events.

\begin{lemma}\label{l:3}  Let ${\cal L}=\{B_1,\ldots,B_N\}$ be a finite  collection of measurable subsets
of $\S.$  Then \[\lim_{n\to\infty} P_n(\cap_{B_i\in {\cal
L}'}B_i)\to P(\cap_{B_i\in {\cal L}'}B_i)\] for all the subsets
${\cal L}'\subseteq {\cal L}$ if and only if
 \[\lim_{n\to\infty} P_n(\cup_{B_i\in {\cal
L}'}B_i)\to P(\cup_{B_i\in {\cal L}'}B_i)\] for all the subsets
${\cal L}'\subseteq {\cal L}$
\end{lemma}
\begin{proof}  If the convergence holds for intersections, it
holds for unions because of the inclusion-exclusion principle. If
the convergence holds for unions, it holds for intersections
because of the inclusion-exclusion principle and induction in the
number of sets in $\cal L$.

\end{proof}

The following two statements follow from Corollary~\ref{c:1} and
Theorem~\ref{t:1} respectively.

\begin{corollary}\label{c:2} Let $\{P_n\}_{n=1,2,\ldots}$ be a sequence of probability measures
from $\P(\S)$ and $P\in\P(\S)$.  If for a each open subset $\oo$
of $\S$ there is a  sequence of measurable subsets
$B_1,B_2,\ldots$ of $\oo$ such that:

(i) $\oo=\cup_{i=1}^\infty B_j,$

(ii) $\lim_{n\to\infty}P_n(\cap_{j=1}^k B_{i_j})= P(\cap_{j=1}^k
B_{i_j})$  for all $\{B_{i_1},B_{i_2},\ldots,B_{i_k}\}\subseteq\{B_1,B_2,\ldots\},$ $k=1,2,\ldots,$
 then $P_n\Wc P$.

\end{corollary}
\begin{proof} In view of Lemma~\ref{l:3}, for each open subset $\oo$ of $\S$ condition (ii) implies
 that  $\lim_{n\to\infty}P_n(\cup_{j=1}^k B_j)= P(\cup_{j=1}^k
B_j)$ for all $k=1,2,\ldots,$ and according to Corollary~\ref{c:1}
these equalities imply that $P_n\Wc P$. \end{proof}
\begin{corollary}\label{cor:1(1space)}
Let  $\{P_{n}\}_{n=1,2,\ldots}$ be a sequence of probability
measures from $\P(\S)$ and $P\in\P(\S)$. If the topology on $\S$
has a countable base $\tau_b$ such that $P_n(\oo)\to P(\oo)$ for
each finite intersection $\oo=\cap_{i=1}^ k {\oo}_{i}$    with
$\oo_{i}\in\tau_b,$ $i=1,2,\ldots,k,$ 
then $P_n\Wc P$.
\end{corollary}
\begin{proof}
 In view of Lemma~\ref{l:3}, $\lim_{n\to\infty}P_n(\oo^*)= P(\oo^*)$ for each finite union
$\oo^*=\cup_{i=1}^k {\oo}_{i}$   with $\oo_{i}\in\tau_b,$
$k=1,2,\ldots\ .$  Theorem~\ref{t:1} implies that $P_n\Wc P$.
\end{proof}

The following example demonstrates that the assumptions of
Corollary~\ref{cor:1(1space)} does not imply that $P_n\Sc P$.
\begin{example}\label{ex2}{\rm  Let $\S=\R^1$, $P$ be a
deterministic measure concentrated at the point $a=\sqrt 2,$  and
$P_n$ be deterministic measures concentrated at the points
$a_n={\sqrt 2}+n^{-1},$ $n=1,2,\ldots \ .$  Since $a_n\to a,$ then
$P_n\Wc P$ as $n\to \infty.$ 
Let $\tau_B$ be the family consisting
of an empty set, $\R^1,$ and of all the open intervals on $\R^1$ with
rational ends.  Then $\tau_b$ is a countable base of the topology
on $\R^1$ generated by the Euclidean metric.  Observe that
$\oo_1\cap\oo_2\in\tau_b$ for all $\oo_1,\oo_2\in\tau_b$, and
$\lim_{n\to\infty} P_n((b_1,b_2))={\bf I}\{a\in (b_1,b_2)\}=P((b_1,b_2))$, for any rational $b_1<b_2$.
Thus the assumptions of Corollary~\ref{cor:1(1space)} hold.
However, of course, it is not true that $P_n\Sc P,$ because
$P_n(\{a\} )=0$ for all $n=1,2,\ldots,$ but $P(\{a\})=1.$}  \hfill$\Box$
\end{example}

\begin{corollary}\label{c:3}
Let  $\{P_{n}\}_{n=1,2,\ldots}$ be a sequence of probability
measures from $\P(\S)$ and $P\in\P(\S)$. If the topology on $\S$
has a countable base $\tau_b$ such that $P_n(\oo)\to P(\oo)$ for
each finite intersection $\oo=\cap_{i=1}^ k {\oo}_{i}$ with
$\oo_{i}\in\tau_b,$ $i=1,2,\ldots,k,$ and, in addition, for any
close set $C\subseteq\S$ there is a  sequence of
measurable subsets $B_1,B_2,\ldots$ of $C$ such that
$C=\cup_{i=1}^\infty B_j$ and condition (ii) of
Corollary~\ref{c:2} holds,
then $P_n\Sc P$.
\end{corollary}
\begin{proof}  Let $\oo$ be an arbitrary open subset.  In view of
Corollary~\ref{cor:1(1space)}, the properties of the base $\tau_b$
imply that $P_n\Wc P$.  Therefore
\begin{equation}\label{eq3.2nn} \ilim_{n\to\infty} P_n(\oo)\ge P(\oo).\end{equation}  Let
$C=\oo^c.$ Condition (ii) of Corollary~\ref{c:2} and
Lemma~\ref{l:3}  imply that $\lim_{n\to\infty}P_n(\cup_{j=1}^k
B_j)=P(\cup_{j=1}^k B_j)$ for all $k=1,2,\ldots.$  In view of
Lemma~\ref{l:1},
\begin{equation}\label{eq3.3nn} \ilim_{n\to\infty} P_n(\oo^c)\ge P(\oo^c).\end{equation}
Inequalities (\ref{eq3.2nn}) and (\ref{eq3.3nn}) imply that
$\lim_{n\to\infty} P_n(\oo)= P(\oo).$  Since $\oo$ is an arbitrary
open subset of $\S$, Theorem~\ref{t3} implies that $P_n\Sc P$.
\end{proof}

\section{Continuity of Transition Probabilities}\label{3A}

For a Borel subset $S$ of a metric space $(\S,\rho)$, where $\rho$
is a metric, consider the metric space $(S,\rho)$.  A set $B$ is called open
(closed,  compact) in $S$ if $B\subseteq S$ and $B$ is open (closed,
compact) in $(S,\rho)$. Of course, if $S=\S$, we omit
``in $\S$''. Observe that, in general, an open (closed, compact) set
in $S$ may not be open (closed, compact). Open sets in $S$ form the topology on $S$ defined by the restriction of metric $\rho$ on $S$.

For  metric spaces $\S_1$ and $\S_2$, a (Borel-measurable) \textit{stochastic
kernel} (sometimes called transition probability) $R(ds_1|s_2)$ on $\S_1$ given $\S_2$ is a mapping $R(\,\cdot\,|\,\cdot\,):\B(\S_1)\times \S_2\to [0,1]$, such that $R(\,\cdot\,|s_2)$ is a
probability measure on $\S_1$ for any $s_2\in \S_2$, and $R(B|\,\cdot\,)$ is a Borel-measurable function on $\S_2$ for any Borel set $B\in\B(\S_1)$. A
stochastic kernel $R(ds_1|s_2)$ on $\S_1$ given $\S_2$ defines a Borel measurable mapping
$s_2\to R(\,\cdot\,|s_2)$ of $\S_2$ to the metric space
$\P(\S_1)$ endowed with the topology of weak convergence.
A stochastic kernel
$R(ds_1|s_2)$ on $\S_1$ given $\S_2$ is called
\textit{weakly continuous (setwise continuous, continuous in  the total variation)}, if $R(\,\cdot\,|s^{(n)})$ converges weakly (setwise, in
 the total  variation) to $R(\,\cdot\,|s)$ whenever $s^{(n)}$ converges to $s$
in $\S_2$.

In the rest of this section,  $\S_1$, $\S_2$ and $\S_3$ are Borel subsets of Polish (complete separable metric) spaces, and $P$ is a
stochastic kernel on $\S_1\times\S_2$ given $\S_3$. The following statement follows from Corollary~\ref{cor:1(1space)}.  As follows from Lemma~\ref{l:3}, the continuity of finite intersection in the condition of Corollary~\ref{teor:2} can be replaced with the assumption that probabilities of finite unions are continuous.

\begin{corollary}\label{teor:2}
If the topology on $\S_i$, $i=1,2$, has a countable base $\tau_b^i$
such that $P(\oo_1\times\oo_2|\,\cdot\,)$ is continuous on $\S_3$
for each finite intersections $\oo_i=\cap_{j=1}^ N {\oo}^{j}_i$ with
$\oo^{j}_i\in\tau_b^i,$ $j=1,2,\ldots,N,$ $i=1,2$, then the
stochastic kernel $P$ on $\S_1\times\S_2$ given $\S_3$ is weakly
continuous.
\end{corollary}

\begin{proof} 
Let
$\tau_b^{1,2}:=\{\oo'_1\times\oo'_2:\, \oo'_i\in\tau_b^i,\ i=1,2\}$.
Note that $\tau_b^{1,2}$ is a countable base of the topology on
$\S_1\times \S_2$ defined as the product of the topologies on
$\S_1$ and $\S_2.$
Observe that   $\cap_{j=1}^N \left(\oo_1^j\times
\oo_2^j\right)= \left(\cap_{j=1}^N\oo_1^j\right)\times
\left(\cap_{j=1}^N\oo_2^j\right)$ 
for any finite
tuples of open sets $\{\oo_i^j\}_{j=1}^{N}$ from $\tau_b^i,$ $i=1,2.$
 Denote $\oo_i=\cap_{j=1}^N\oo_i^j$ for $i=1,2.$
By the assumption of Corollary~\ref{teor:2}, $P_n(\oo_1\times\oo_2|\cdot)$ is continuous on $\S_3.$
This means that the assumption of Corollary~\ref{cor:1(1space)} holds for the base $\tau_b^{1,2}.$  Corollary~\ref{cor:1(1space)} implies that the stochastic kernel $P$ on $\S_1\times\S_2$
given $\S_3$ is weakly continuous.
%
%
\end{proof}

Let  $\F(\S)$ and $\C(\S)$ be respectively the spaces of all
real-valued functions and all bounded continuous functions defined
on the metric space $\S$. A subset $\mathcal{A}_0\subseteq \F(\S)$
is said to be \textit{equicontinuous at a point $s\in\S$}, if $
\sup\limits_{f\in\mathcal{A}_0}|f(s')-f(s)|\to 0$ as $s'\to s. $
If a family $\mathcal{A}_0\subseteq \F(\S)$ is equicontinuous at each point $s\in\S,$
 it is called equicontinuous on $\S.$
A
subset $\mathcal{A}_0\subseteq \F(\S)$ is said to be
\textit{uniformly bounded}, if there exists a constant $M<+\infty $ such that $ |f(s)|\le M$ for all $s\in\S$ and  for
all $f\in\mathcal{A}_0. $  Obviously, if a subset
$\mathcal{A}_0\subseteq \F(\S)$ is equicontinuous at all the
points $s\in\S$ and uniformly bounded, then
$\mathcal{A}_0\subseteq \C(\S).$

\begin{theorem}\label{kern}{\rm(Feinberg et al.
\cite[Theorem~5.2]{FKZ}).} Let $\S_1$, $\S_2$, and $\S_3$ be
arbitrary metric spaces, $P(ds_2|s_1)$ be a weakly continuous
stochastic kernel on $\S_2$ given $\S_1$, and a subset
$\mathcal{A}_0\subseteq \C(\S_2\times\S_3)$ be equicontinuous at all
the points $(s_2,s_3)\in\S_2\times\S_3$ and uniformly bounded. If
$\S_2$ is separable, then for every open set $\oo$ in $\S_2$ the
family of functions defined on $\S_1\times\S_3$,
\[
\mathcal{A}_\oo=\left\{(s_1,s_3)\to\int_{\oo}f(s_2,s_3)P(ds_2|s_1)\,:\,
f\in\mathcal{A}_0\right\},
\]
is equicontinuous at all the points $(s_1,s_3)\in\S_1\times\S_3$ and
uniformly bounded.
\end{theorem}

Further as $\tau(\S)$ we denote the family of all open subsets of a metric space $\S$.
For each $B\in{\cal B}(\S_1)$  consider a family of functions
\[\mathcal{P}_B=\{  s_3\to P(B\times C|s_3):\, C\in \tau(\S_2)\}\] mapping $\S_3$ into $[0,1]$.

\begin{lemma}\label{lem:PB} Let $B\in{\cal B}(\S_1)$. The family of functions $\mathcal{P}_B$
is equicontinuous at a point $s_3 \in \S_3$ if and only if
\begin{equation}
\label{eq:EC}
\sup_{C \in \B(\S_2)} | P(B \times C | s_3^{(n)}) - P(B \times C | s_3)| \to 0 \qquad \mbox{ as }  \qquad s_3^{(n)} \to s_3.
\end{equation}
\end{lemma}
\begin{proof}
According to the definition of the equicontinuity of the family of functions
$\mathcal{P}_B$ at a point, it is sufficient to prove that (\ref{eq:EC}) follows from
\[
\sup_{C \in \tau(\S_2)} | P(B \times C | s_3^{(n)}) - P(B \times C | s_3)|\to 0 \qquad \mbox{ as }  \qquad s_3^{(n)} \to s_3.
\]
Indeed, if $P(B \times \S_2 | s_3)=0$, then $\sup_{C \in \B(\S_2)} | P(B \times C | s_3^{(n)}) - P(B \times C | s_3)|= P(B \times \S_2 | s_3^{(n)})\to P(B \times \S_2 | s_3)=0$ as $s_3^{(n)} \to s_3$, because $\S_2\in\tau(\S_2)$. Otherwise, when $P(B \times \S_2 | s_3)>0$, according to the convergence $P(B \times \S_2 | s_3^{(n)})\to P(B \times \S_2 | s_3)>0$ as $s_3^{(n)} \to s_3$, Theorem~\ref{t5}(ii) applied to the probability measures $C\to P(B \times C | s_3^{(n)})/ P(B \times \S_2 | s_3^{(n)})$ and $C\to P(B \times C | s_3)/ P(B \times \S_2 | s_3)$ from $\P(\S_2)$, where $n$ is rather large, yields that (\ref{eq:EC}) holds, that is, the family of functions $\mathcal{P}_B$
is equicontinuous at a point $s_3 \in \S_3$.
\end{proof}

Let $P'$ be the marginal of $P$ on $\S_2$, that is,
$P'(C|s_3):=P(\S_1\times C|s_3)$, $C\in \B(\S_2)$, $s_3\in \S_3$. There exists a stochastic
kernel $H$ on $\S_1$ given $\S_2\times\S_3$ such that, for all $B\in \B(\S_1), C\in \B(\S_2),s_3\in \S_3$
\begin{equation}
\label{eq:H}
P(B\times C|s_3)=\int_{C}H(B|s_2,s_3)P'(ds_2|s_3);
\end{equation}
Bertsekas and Shreve~\cite[Proposition~7.27]{BS}. Moreover, for each $s_3 \in \S_3$, the distribution $H(\,\cdot\, | s_2, s_3)$ is $P'(\,\cdot\,|s_3)$-a.s.\, unique in $s_2$, that is, if $H_1$ and $H_2$ satisfy \eqref{eq:H} then $P'(C^*| s_3) = 0$, where $C^* := \{ s_2 \in \S_2 : H_1(B|s_2,s_3) \not = H_2(B|s_2,s_3) \mbox{ for some } B \in \B(\S_1)\}$; Bertsekas and Shreve~\cite[Corollary~7.27.1]{BS}. 

\begin{theorem}
\label{mainthkern} Let the topology on $\S_1$ have a countable base
$\tau_b $ satisfying the following two conditions: \begin{itemize} \item[(i)] $\S_1\in \tau_b,$ \item[(ii)] for each finite
intersection $\oo = \bigcap_{i = 1}^N \oo^{i}$ of sets
$\oo^{i} \in \tau_b$, $i = 1,2,\ldots, N$, the family of
functions $\mathcal{P}_{\oo}$ is equicontinuous at a point
$s_s\in \S_3$. \end{itemize} Then, for any sequence $\{s_3^{(n)}\}_{n =
1,2,\ldots}$ from $\S_3$ converging to $s_3\in\S_3$, there
exists a subsequence $\{n_k\}_{k=1,2,\ldots}$ and a set $C^* \in
\B(\S_2)$ such that
\begin{equation}
\label{result}
P'(C^* | s_3)  = 1  \mbox{ and }\, H(\,\cdot\, | s_2, s_3^{(n_k)} ) \mbox{ converges weakly to } H(\,\cdot\, | s_2, s_3) \mbox{ for all } s_2 \in C^*\  as\  k\to\infty.
\end{equation}
\end{theorem}

\begin{remark}\label{rem:1}
According to Lemma~\ref{l:3}, a countable base $\tau_b$ in
Theorem~\ref{mainthkern} can be assumed to be closed with respect to the finite
unions instead of finite intersections.
\end{remark}

Theorem~\ref{mainthkern} implies the following two corollaries.  The proof of Theorem~\ref{mainthkern} is provided after the proof of Lemma~\ref{b1b2}.

\begin{corollary}
\label{Cormainkern}
If for each open subset $\oo$ of $\S_1$  the family of functions $\mathcal{P}_{\oo}$ is equicontinuous at a point $s_3\in \S_3$, then for any sequence $\{s_3^{(n)}\}_{n = 1,2,\ldots}$ from $\S_3,$ that converges to $s_3\in \S_3$, there exists a subsequence $\{n_k\}_{k=1,2,\ldots}$ and
a set $C^* \in \B(\S_2)$ such that \eqref{result} holds.
\end{corollary}
\begin{proof}
The statement of the corollary follows immediately from Theorem~\ref{mainthkern}. Indeed, the family of functions $\mathcal{P}_{\oo}$ is equicontinuous on $ \S_3$ for each open set $\oo$ of $\S_1.$ Since $\S_1$ is a separable metric space, each countable base of the topology on $\S_1$ satisfies assumptions of Theorem~\ref{mainthkern}.
\end{proof}

Observe that for a stochastic kernel $P$ on $\S_1\times\S_2$ given $\S_3,$
equicontinuity at a point $s_3\in \S_3$ of the family
of functions $\mathcal{P}_\oo$ for all open subsets $\oo$ in $\S_1$ is a weaker assumption than 
continuity in the total variation of $P$ on $\S_1\times\S_2$ given $\S_3$ at the point $s_3.$ 
Equicontinuiuty of the family of functions $\mathcal{P}_{\S_1}$ at a point $s_3\in\S_3$ is equivalent to the continuity in the total variation of the stochastic kernel $P^\prime$ on $\S_2$ given $\S_3$ at the point $s_3.$

\begin{corollary}\label{cor:1}
Let assumptions of Theorem~\ref{mainthkern} hold. If the setwise convergence takes place in \eqref{result} instead of the weak convergence,
then the stochastic kernel $P$ on $\S_1\times\S_2$ given $\S_3$ is setwise continuous.
\end{corollary}
\begin{proof}
According to Theorem~\ref{t3}, if the stochastic kernel $P$ on $\S_1\times\S_2$ given $\S_3$ is not setwise continuous, then there exist $\varepsilon>0$, a nonempty open subset $\oo$ of $\S_1\times\S_2$, and a sequence $\{s_3^{(n)}\}_{n=1,2,\ldots}$ that converges to some $s_3\in\S_3$ such that
\begin{equation}\label{AB}
|P(\oo|s_3^{(n)})-P(\oo|s_3)|\ge \varepsilon\mbox{ for each } n=1,2,\ldots\ .
\end{equation}

Let $\oo_2$ be the projection of $\oo$ on $\S_2$ and $\oo_{(s_2)}:=\{s_1\in \S_1\,:\, (s_1,s_2)\in \oo\}$ be the cut of $\oo$ at $s_2\in \oo_2$.
Since $\oo$ is an open set, the sets $\oo_2$ and $\oo_{(s_2)}$ are open.
Since $P'(ds_2|s_3^{(n)})$ converges in the total variation to $P'(ds_2|s_3),$ for any $s_3\in \S_3$
\begin{equation}\label{eq:A}
\left|\int_{\oo_2}H(\oo_{(s_2)}|s_2,s_3^{(n)})P'(ds_2|s_3^{(n)})-\int_{\oo_2}H(\oo_{(s_2)}|s_2,s_3^{(n)})P'(ds_2|s_3)\right|\to 0 \mbox{ as }n\to \infty.
\end{equation}
According to the assumptions of Corollary~\ref{cor:1}, there exists a set $C^* \in \B(\S_2)$ and a subsequence $\{s_3^{(n_k)}\}_{k=1,2,\ldots}$ of $\{s_3^{(n)}\}_{n=1,2,\ldots}$ such that $P'(C^* | s_3)=1$   and $H(\,\cdot\, | s_2, s_3^{(n_k)})$ converges setwise to $H(\,\cdot\, | s_2, s_3)$ for any $s_2 \in C^*$. In particular, $H(\oo_{(s_2)}|s_2,s_3^{(n_k)})\to H(\oo_{(s_2)}|s_2,s_3)$ for any $s_2\in C^*$. Therefore, the dominated convergence theorem yields
\begin{equation}\label{eq:B}
\int_{\oo_2}\left|H(\oo_{(s_2)}|s_2,s_3^{(n_k)})-H(\oo_{(s_2)}|s_2,s_3)\right|P'(ds_2|s_3)\to 0 \mbox{ as }k\to \infty.
\end{equation}
Formulae \eqref{eq:A} and \eqref{eq:B} imply that as $k\to\infty$
\[
P(\oo|s_3^{(n_k)})=\int_{\oo_2}H(\oo_{(s_2)}|s_2,s_3^{(n_k)})P'(ds_2|s_3^{(n)})\to \int_{\oo_2}H(\oo_{(s_2)}|s_2,s_3)P'(ds_2|s_3)=P(\oo|s_3).
\]
This contradicts  (\ref{AB}). Thus the stochastic kernel $P$ on $\S_1\times\S_2$ given $\S_3$ is setwise continuous.
\end{proof}

The proof of Theorem~\ref{mainthkern} uses several auxiliary results.

\begin{lemma}{\rm(Feinberg et.\,al~\cite[Theorem~5.5]{FKZ}).}
\label{setwise} Let $h$ and $\{h^{(n)}\}_{n = 1,2,\ldots}$ be
Borel-measurable uniformly bounded real-valued functions defined on
a metric space $\S$ and let $\{\mu^{(n)}\}_{n = 1,2,\ldots}$ be a
sequence of probability measures from $\P(\S)$ that converge in  the
total  variation to the measure $\mu\in\P(\S)$. If
\begin{equation}\label{sw1}
\sup_{C\in\B(\S)}\left|\int_{C}h^{(n)}(s)\mu^{(n)}(ds)-
\int_{C}h(s)\mu(ds)\right|\to 0 \quad {\rm as}\quad
n\to\infty,
\end{equation}
then $\{h^{(n)}\}_{n = 1,2,\ldots}$ converges in probability $\mu $ to  $h$ as
$n\to\infty$, and therefore there is a subsequence $\{n_k\}_{k = 1,2,\ldots}$ such that
$\{h^{(n_k)}\}_{k = 1,2,\ldots}$ converges $\mu$-almost surely to $h$.
\end{lemma}

 Let $\A_1$ be the family of all subsets of $\S_1$ that are finite unions of sets from the countable base $\tau_b$ of the topology on $\S_1$ satisfying the conditions of
 Theorem~\ref{mainthkern}, and $\A_2$ be the family of all subsets $B$ of $\S_1$ such that $B = \tilde{\oo} \setminus \oo'$ with $\tilde{\oo} \in \tau_b$ and $\oo' \in \A_1$.
\begin{lemma}
\label{b1b2}
Let the assumptions of Theorem~\ref{mainthkern} hold for some $s_3 \in \S_3$. Then, for any subset $B \in \A_2$, the family of functions $\mathcal{P}_{B}$ is equicontinuous at  the point $s_3 \in \S_3$.
\end{lemma}
\begin{proof} Fix an arbitrary $s_3\in\S_3.$
Observe that, if  for all $\oo \in \A_1$ the family of functions $\mathcal{P}_{\oo}$ is equicontinuous at  the point $s_3 \in \S_3$, then for any subset $B = \tilde{\oo} \setminus \oo'$ of $\S_1$ with $\tilde{\oo} \in \tau_b$ and $\oo' \in \A_1$,  the family of functions $\mathcal{P}_{B}$ is equicontinuous at  the point $s_3 \in \S_3$. Indeed, according to Lemma~\ref{lem:PB}, for all $s_3, s_3^{(n)} \in \S_3$, $n = 1,2,\ldots,$ such that $s_3^{(n)} \to s_3$ as $n \to \infty$,
\begin{gather*}
\sup_{C \in \B(\S_2)}|P(B \times C | s_3^{(n)}) - P(B \times C | s_3)| = \sup_{C \in \B(\S_1)}|P((\tilde{\oo}\setminus \oo')\times C  | s_3^{(n)}) - P( (\tilde{\oo}\setminus \oo')\times C  | s_3)|\\
\le \sup_{C \in \B(\S_2)}|P(\oo' \times C   | s_3^{(n)}) - P( \oo' \times C | s_3)| + \sup_{C \in \B(\S_2)}|P(( \tilde{\oo} \cup \oo') \times C   | s_3^{(n)}) - P((\tilde{\oo} \cup \oo') \times C   | s_3)|.
\end{gather*}
The above inequality, the assumption that \eqref{eq:EC} holds for all $\oo \in \A_1$ and for all  $s_3, s_3^{(n)} \in \S_3$, $n = 1,2,\ldots$, such that $s_3^{(n)}  \to s_3$ as $n \to \infty$, and the property that if $\oo' \in \A_1$ then $\tilde{\oo} \cup \oo' \in \A_1$ for all $\tilde{\oo} \in \tau_b$ imply that \eqref{eq:EC} holds for any subset $B \in \A_2$,
that is, the family of functions $\mathcal{P}_{B}$ is equicontinuous at  the point $s_3 \in \S_3$. The rest of the proof establishes that,  for each $\oo \in \A_1$, the family of functions $\mathcal{P}_{\oo}$ is equicontinuous at  the point $s_3 \in \S_3$.

Let $\tau_b = \{\oo^{(j)}\}_{j = 1,2,\ldots}$. Consider an arbitrary $\oo \in \A_1$.  Then $\oo = \cup_{i = 1}^{N} \oo^{(j_i)}$ for some $N = 1,2,\ldots$, where $\oo^{(j_i)} \in \tau_b$, $i = 1,2,\ldots, N$. Let $\A^N = \{\cap_{m = 1}^k \oo^{(i_m)}: \{i_1, i_2, \ldots, i_k\} \subseteq \{j_1, j_2, \ldots j_N\}\}$ be the finite set of possible intersections of $\oo^{(j_1)}, \ldots, \oo^{(j_N)}$. The principle of inclusion-exclusion implies that for $\oo = \cup_{i = 1}^{N} \oo^{(j_i)}$, $C\in \S_2$, and $s_3, s_3^{(n)} \in \S_3$,
\[|P(\oo\times C  | s_3) - P(\oo\times C  | s_3^{(n)})| \le \sum_{D \in \A^N} |P( D\times C | s_3) - P( D \times C| s_3^{(n)})|.\]
The above inequality and the assumption of Theorem~\ref{mainthkern} regarding finite intersections of the elements of the base $\tau_b$ imply that, for each $\oo \in \A_1$, the family of functions $\mathcal{P}_{\oo}$ is equicontinuous at  the point $s_3 \in \S_3$.
\end{proof}

\begin{proof}[Proof of Theorem~\ref{mainthkern}]
Let $\{s_3^{(n)}\}_{n =
1,2,\ldots}$ be a sequence from $\S_3$ that converges to $s_3\in \S_3$. According to Theorem~\ref{t1},
\eqref{result} holds if there exists a subsequence $\{n_m\}_{m=
1,2,\ldots}$ and a set $C^* \in \B(\S_2)$ such that for all open subsets $\oo$ in $\S_1$
\begin{equation}
\label{special}
P'(C^* | s_3)  = 1 \quad \mbox{ and } \quad \ilim\limits_{m\to\infty} H( \oo \,|\, s_2, s_3^{(n_m)} ) \ge H(\oo \,|\, s_2, s_3) \quad \mbox{ for all } \quad s_2 \in C^*.
\end{equation}
 The rest of the proof establishes the existence of a subsequence  $\{s_3^{(n_m)}\}_{m= 1,2,\ldots}$
of the sequence $\{s_3^{(n)}\}_{n = 1,2,\ldots}$ and a set $C^* \in \B(\S_2)$ such that \eqref{special} holds for each open subset $\oo$ of $\S_1$.
%
%
%
%

Let $\A_1$ and $\A_2$ be the families of subsets of $\S_1$ as defined before Lemma~\ref{b1b2}. Observe that: (i) both $\A_1$ and $\A_2$ are countable, (ii) every open subset $\oo$ of $\S_1$ can be represented as
\begin{equation}
\label{partition}
\oo = \bigcup_{j = 1,2,\ldots} \oo^{(j,1)}  = \bigcup_{j = 1,2,\ldots} B^{(j,1)}, \quad  \mbox{ for some } \quad \oo^{(j,1)} \in \tau_b, j = 1,2,\ldots,
\end{equation}
where $B^{(j,1)} = \oo^{(j,1)} \setminus (\cup_{i = 1}^{j-1} \oo^{(i,1)})$ are disjoint elements of $\A_2$ (it is allowed that $\oo^{(j,1)} = \emptyset$ or $B^{(j,1)} = \emptyset$ for some $j = 1,2,\ldots$).

To prove \eqref{special} for all open subsets $\oo$ of $\S_1$, we first show that \eqref{special} holds for all $\oo \in \A_2$.
From Lemmas~\ref{lem:PB}, \ref{b1b2} and \eqref{eq:H},
\begin{equation}
\label{ECO}
\lim_{n \to \infty}\sup_{C\in\B(\S_2)}\left|\int_{C} H(B|s_2, s_3^{(n)}) P'(ds_2|s_3^{(n)})- \int_{C}H(B|s_2, s_3) P'(ds_2|s_3)\right| = 0, \quad B \in \A_2. 
\end{equation}
Since the set $\A_2$ is countable, let $\A_2 := \{B^{(j)}: j = 1,2,\ldots\}$. Choose a subsequence $\{s_3^{(n_k)}\}_{k= 1,2,\ldots}$ of the sequence $\{s_3^{(n)}\}_{n = 1,2,\ldots}$. Denote $s^{(n,0)}=s_3^{(n)}$ for all $n=1,2,\ldots\ .$ For $j = 1,2,\ldots$,  from \eqref{ECO}, Lemma~\ref{setwise}, applied with $s = s_2$, $h^{(n)}(s) = H(B^{(j)} |s_2, s^{(n, j-1)})$, $\mu^{(n)}(\cdot) = P'(\,\cdot\,| s^{(n, j-1)})$, $h(s) = H(B^{(j)} |s_2, s_3)$, and $\mu(\cdot) = P'(\,\cdot\,| s_3)$, there exists a subsequence $\{s^{(n, j)}\}_{n = 1,2,\ldots}$ of the sequence $\{s^{(n, j-1)}\}_{n = 1,2,\ldots}$ and a set $C^*_j\in \B(\S_2)$ such that
\begin{equation}
\label{B-i}
\lim_{n \to \infty}H( B^{(j)} | s_2, s^{(n, j)}) = H(B^{(j)}|s_2, s_3) \quad \mbox{ for all } \quad s_2 \in C_j^*.
\end{equation}
Let $C^*=\cap_{j=1,2,\ldots} C_j^*$. Observe that $P'(C^*|s_3)=1$. Let $s_3^{(n_m)}=s^{(m, m)},$ $m=1,2,\ldots\ .$ As follows from Cantor's diagonal argument,
\eqref{special} holds with $\oo=B^{(j)}$ for all $j = 1, 2, \ldots\ .$ In other words, \eqref{special} is proved for all $\oo \in \A_2$.

Let $\oo$ be an arbitrary open set in $\S_1$ and $B^{(1,1)}, B^{(2,1)}, \ldots$ be disjoint elements of $\A_2$ satisfying \eqref{partition}. Then the countable additivity of probability measures implies that, for all $s_2 \in C^*$,
\begin{multline*}
\begin{aligned}
\ilim\limits_{m\to\infty}H(\oo|s_2, s_3^{(n_m)}) &=\ilim\limits_{m\to\infty}\sum_{j=1,2,\ldots} H(B^{(j,1)}|s_2, s_3^{(n_m)}) \ge \sum_{j=1,2,\ldots}\ilim\limits_{m\to\infty}H(B^{(j,1)}|s_2, s_3^{(n_m)}) \\
&=\sum_{j=1,2,\ldots} H(B^{(j,1)}|s_2, s_3)=H(\oo|s_2, s_3).
\end{aligned}
\end{multline*}
Therefore, \eqref{special} holds for all open subsets $\oo$ in $\S_1$.
\end{proof}

\begin{example}\label{exa:MDM} (Stochastic kernel $P$ on $\S_1\times\S_2$ given $\S_3$ satisfies assumptions of Theorem~\ref{mainthkern}, but it is not setwise continuous and it does not satisfy the assumption  of Corollary~\ref{Cormainkern}.)
{\rm Let $\S_1=\R^1$, $\S_2=\{1\}$,
$\S_3=\{1^{-1},2^{-1},\ldots,0\}$, $\tau_B$ be the family consisting
of an empty set, $\R^1,$ and of all the open intervals on $\R^1$ with
rational ends, and $P(B\times C|s_3)=\I\{\sqrt{2}+s_3\in B\}\I\{1\in
C\}$, $B\in\B(\S_1)$, $C\in\B(\S_2)$. Then $P'(C)=\I\{1\in C\}$,
$H(B|s_2,s_3)=\I\{\sqrt{2}+s_3\in B\}$, $B\in\B(\S_1)$,
$C\in\B(\S_2)$. Let
$\tau_b$ be the countable base of the topology on $\R^1$ generated by
the Euclidean metric described in Example~\ref{ex2}. The family $\tau_b$ is closed under finite
intersections, and for any $\oo\in \tau_b$ the family of functions
$\mathcal{P}_\oo$ is equicontinuous at all the points $s_3 \in
\S_3$. Therefore, assumptions of Theorem~\ref{mainthkern} hold.

Note that the function $P(B\times C|s_3)$ is not continuous at the
point $s_3=0,$ when $B=\R^1\setminus\{\sqrt{2}\}$ and $C=\S_3$.
Therefore, the family $\mathcal{P}_B$ is not equicontinuous at
the point $s_3 =0,$ and the assumption of
Corollary~\ref{Cormainkern} do not hold. Moreover, the sequence
$\{H(B|1,\frac1n)\}_{n=1,2,\ldots}$ (and any its subsequence) does
not converge to $H(B|1,0)$ and, therefore, the setwise convergence assumption from
Corollary~\ref{cor:1} do not hold. }\hfill$\Box$
\end{example}

\section{Partially Observable Markov Decision Processes}\label{S4}
Convergence properties of probability measures and relevant
continuity properties of transition probabilities are broadly used
in  mathematical methods of stochastic control.  In this section,
we describe the results for a Bayesian sequential decision model,
a POMDP.  For POMDPs, posterior probabilities of states of the
process form sufficient statistics; see e.g.,
Hern\'{a}ndez-Lerma~\cite[p. 89]{HL}.  In terms of Markov Decision
Processes, this well-known fact means that it is possible to
construct an MDP, called a Completely Observable Markov Decision
Process (COMDP), whose state space is the space of probability
measures on the original state space.    If an optimal policy is
found for a COMDP, it is easy to compute an optimal policy for the
original POMDP.  However, except the cases of finite state spaces
(Smallwood and Sondik~\cite{SS}, Sondik~\cite{So}), MDMIIs with
transition probabilities having densities (Rieder~\cite{Ri},
B\"auerle and Rieder~\cite[Chapter 5]{BR}),  models explicitly
defined by equations for continuous random variables
(Striebel~\cite{St}, Bensoussan~\cite{Be}), and numerous
particular problems studied in the literature, until recently very
little had been known about the existence and characterizations of
optimal policies for POMDPs and their COMDPs. The main difficulty
is that the transition probability for a COMDP is defined via the
Bayes formula presented in formula (\ref{3.1}) below, and the
explicit forms of the Bayes formula are known either for discrete
events or for continuous random variables; see Shityaev~\cite[p. 231]{Sh}.
Recently Feinberg et al.~\cite{FKZ} established sufficient
conditions for the existence of optimal policies and their
characterization for POMDPs with Borel state, action, and
observation spaces.

In this section we define POMDPs,  explain their reduction to
COMDPs,  survey some of the results from Feinberg et
al.~\cite{FKZ}, and present the condition on joint distributions of posterior distributions and observations that implies weak continuity of
transition probabilities for the COMDP. In the following section, we describe a more
particular model, the MDMII, and apply Corollary~\ref{cor:1} and results of this section to
it.

Let $\X$, $\Y$, and $\A$ be 
Borel subsets of Polish spaces, 
$P(dx'|x,a)$ be a stochastic kernel on
$\X$ given $\X\times\A$, $Q(dy| a,x)$ be a stochastic kernel on
$\Y$ given $\A\times\X$, $Q_0(dy|x)$ be a stochastic kernel on
$\Y$ given $\X$, $p$ be a probability distribution on $\X$,
$c:\X\times\A\to {\bar\R}^1=\mathbb{R}^1\cup\{+\infty\}$ be a bounded below
Borel function on $\X\times\A.$ 

A 
{\it POMDP} is specified by a tuple $(\X,\Y,\A,P,Q,c)$, where $\X$
is the \textit{state space}, $\Y$ is the \textit{observation set},
$\A$ is the \textit{action} \textit{set}, $P(dx'|x,a)$ is the
\textit{state transition law}, $Q(dy| a,x)$ is the
\textit{observation stochastic kernel}, $c:\X\times\A\to {\bar\R}^1$ is the
\textit{one-step cost}.

The partially observable Markov decision process evolves as
follows: (i) at time $t=0$, the initial unobservable state $x_0$
has a given prior distribution $p$; (ii) the initial observation
$y_0$ is generated according to the initial observation stochastic kernel
$Q_0(\,\cdot\,|x_0)$; (iii) at each time epoch $t=0,1,\ldots,$ if
the state of the system is $x_t\in\X$ and the decision-maker
chooses an action $a_t\in \A$, then the cost $c(x_t,a_t)$ is
incurred; (iv) the system moves to a state $x_{t+1}$ according to
the transition law $P(\,\cdot\,|x_t,a_t)$, $t=0,1,\ldots$; 
(v) an observation $y_{t+1}\in\Y$ is generated by the
observation stochastic kernel  $Q(\,\cdot\,|a_t,x_{t+1})$, $t=0,1,\ldots\ .$ 

Define the \textit{observable histories}: $h_0:=(p,y_0)\in \H_0$
and $h_t:=(p,y_0,a_0,\ldots,y_{t-1}, a_{t-1}, y_t)\in\H_t$ for all
$t=1,2,\dots$,
where $\H_0:=\P(\X)\times \Y$ and $\H_t:=\H_{t-1}\times \A\times
\Y$ if $t=1,2,\dots$. A \textit{policy} $\pi$ 
for the POMDP is defined as a sequence
$\pi=\{\pi_t\}_{t=0,1,\ldots}$ 
of stochastic kernels $\pi_t$ on $\A$ given $\H_t$. A policy $\pi$ 
is called \textit{nonrandomized}, if each probability measure
$\pi_t(\,\cdot\,|h_t)$ is concentrated at one point.  The
\textit{set of all policies} is denoted by $\Pi$. The Ionescu
Tulcea theorem (Bertsekas and Shreve \cite[pp. 140-141]{BS} or
Hern\'andez-Lerma and Lasserre \cite[p.178]{HLerma1}) implies that
a policy $\pi\in \Pi$ and an initial distribution $p\in \P(\X)$,
together with the stochastic kernels $P$, $Q$ and $Q_0$, determine
a unique probability measure $P_{p}^\pi$ on the set of all
trajectories
$
(\X\times\Y\times \mathbb{A})^{\infty}$ endowed with the
$\sigma$-field defined by the products of Borel
$\sigma$-fields  $\B(\X)$, $\B(\Y)$, and $\B(\mathbb{A})$. 
The expectation with respect to this probability measure is
denoted by $\E_{p}^\pi$.

For a finite horizon $T=0,1,...,$ 
the \textit{expected total discounted costs} are
\begin{equation}\label{eq1}
V_{T,\alpha}^{\pi}(p):=\mathbb{E}_p^{\pi}\sum\limits_{t=0}^{T-1}\alpha^tc(x_t,a_t),\qquad\qquad
p\in \P(\X),\,\pi\in\Pi,
\end{equation}
where $\alpha\ge 0$ is the discount factor,
$V_{0,\alpha}^{\pi}(p)=0.$ Consider the following assumptions. \vskip 0.9 ex 

\noindent\textbf{Assumption (D)}. $c$ is bounded below on
$\X\times\A$ and
  $\alpha\in (0,1)$.

\noindent\textbf{Assumption (P)}. $c$ is nonnegative on
$\X\times\A$ and $\alpha=1$.\vskip 0.9 ex

When $T=\infty,$ formula (\ref{eq1}) defines the \textit{infinite
horizon expected total discounted cost}, and we denote it by
$V_\alpha^\pi(p).$ For any function $g^{\pi}(p)$, including
$g^{\pi}(p)=V_{T,\alpha}^{\pi}(p)$ and
$g^{\pi}(p)=V_{\alpha}^{\pi}(p)$, define the \textit{optimal
values}
\begin{equation*} g(p):=\inf\limits_{\pi\in \Pi}g^{\pi}(p), \qquad
\ p\in\P(\X).
\end{equation*} 
A policy $\pi$ is called \textit{optimal} for the respective
criterion, if $g^{\pi}(p)=g(p)$ for all $p\in \P(\X).$ For
$g^\pi=V_{T,\alpha}^\pi$, the optimal policy is called
\emph{$T$-horizon discount-optimal}; for $g^\pi=V_{\alpha}^\pi$,
it is called \emph{discount-optimal}.


We recall that a function $c$ defined  on $\X\times\A$ with values in ${\bar \R}^1$ is inf-compact 
if the set $\{(x,a)\in \X\times\A:\, c(x,a)\le \lambda\}$ is
compact  for any finite number $\lambda.$ A function $c$ defined
on $\X\times \A$ with values in ${\bar \R}^1$ is called $\K$-inf-compact on
$\X\times\A$, if for any compact set $K\subseteq\X$, the function
$c:K\times\A\to {\bar \R}^1$ defined on $K\times\A$ is inf-compact;
Feinberg et al.~\cite[Definition 1.1]{FKV, FKN}. According to
Feinberg et al.~\cite[Lemma 2.5]{FKN}, a bounded below function
$c$ is $\K$-inf-compact on the product of metric spaces $\X$ and
$\A$ if and only if it satisfies the following two conditions:

(a) $c$ is lower semi-continuous;

(b) if a sequence $\{x^{(n)} \}_{n=1,2,\ldots}$ with values in
$\X$ converges and its limit $x$ belongs to $\X$ then any sequence
$\{a^{(n)} \}_{n=1,2,\ldots}$ with $a^{(n)}\in \A$,
$n=1,2,\ldots,$ satisfying the condition that the sequence
$\{c(x^{(n)},a^{(n)}) \}_{n=1,2,\ldots}$ is bounded above, has a
limit point $a\in\A.$

For a POMDP $(\X,\Y,\A,P,Q,c)$, consider the  MDP $(\X,\A,P,c)$,
in which all the states are observable. An MDP can be viewed as a
particular POMDP with $\Y=\X$ and $Q(B|a,x)=Q(B|x)={\bf I}\{x\in
B\}$ for all $x\in\X,$ $a\in \A$, and $B\in{\mathcal B}(\X)$. In
addition, for an MDP an initial state is observable.  Thus for an
MDP an initial state  $x$ is considered instead of the initial
distribution $p.$ In fact, this MDP possesses a special property
that action sets at all the states are equal.

It is well known that the analysis and optimization of an POMDP
can be reduced to the analysis and optimization to a specially
constructed MDPs called a COMDP.  The states of the COMDP are
posterior state distributions of the original POMDP. In order to
find an optimal policy for POMDP, it is sufficient to find such a
policy for the COMDP, and then it is easy to construct an optimal
policy for the COMDPs (see Bertsekas and Shreve \cite[Section
10.3]{BS}, Dynkin and Yushkevich \cite[Chapter 8]{DY},
Hern\'{a}ndez-Lerma \cite[p. 87]{HL}, Yushkevich \cite{Yu} or
Rhenius \cite{Rh} for details).  However, little is known about
the existence of optimal policies for COMDPs and how to find them
when the state, observation, and action sets are Borel spaces. The
rest of this section presents recent results from Feinberg et
al.~\cite{FKZ} on the existence optimal policies and their
computation for COMDPs and therefore for POMDPs.

%

Our next goal is to define the transition probability $q$ for the
COMDP presented in (\ref{3.7}).  Given a posterior distribution
$z$ of the state $x$ at time epoch $t=0,1,\ldots$ and given an
action $a$ selected at epoch $t$, denote by $R(B\times C|z,a) $
the joint probability that the state at time $(t+1)$ belongs to
the set $B\in {\mathcal B}(\X)$ and the observation at time $t+1$
belongs to the set $C\in {\mathcal B}(\Y)$,
\begin{equation}\label{3.3}
R(B\times C|z,a):=\int_{\X}\int_{B}Q(C|a,x')P(dx'|x,a)z(dx),\ B\in
\mathcal{B}(\X),\ C\in \mathcal{B}(\Y),\ z\in\P(\X),\ a\in \A.
\end{equation}
Observe that $R$ is a stochastic kernel on $\X\times\Y$ given
${\P}(\X)\times \A$;  see Bertsekas and Shreve \cite[Section
10.3]{BS}, Dynkin and Yushkevich \cite[Chapter 8]{DY},
Hern\'{a}ndez-Lerma \cite[p. 87]{HL},  Yushkevich \cite{Yu}, or
Rhenius \cite{Rh} for details.
 The  probability that the observation $y$ at time $t+1$
belongs to the set $C\in\B(\Y)$, given that at time $t$ the
posterior state probability is $z$ and selected action is $a,$ is
$R'(C|z,a):=R(\X\times C|z,a)$,
$C\in \mathcal{B}(\Y)$, $z\in\P(\X)$, $a\in\A$.
Observe that $R'$ is a stochastic kernel on $\Y$ given
${\P}(\X)\times \A.$ By Bertsekas and Shreve~\cite[Proposition
7.27]{BS}, there exist a stochastic kernel $H$ on $\X$ given
${\P}(\X)\times \A\times\Y$ such that
\begin{equation}\label{3.4}
R(B\times C|z,a)=\int_{C}H(B|z,a,y)R'(dy|z,a),\quad B\in
\mathcal{B}(\X),\  C\in \mathcal{B}(\Y),\ z\in\P(\X),\ a\in \A.
\end{equation}

The stochastic kernel $H(\,\cdot\,|z,a,y)$ defines a measurable
mapping $H:\,\P(\X)\times \A\times \Y \to\P(\X)$, where
$H(z,a,y)(\,\cdot\,)=H(\,\cdot\,|z,a,y).$ For each pair $(z,a)\in
\P(\X)\times\A$, the mapping $H(z,a,\cdot):\Y\to\P(\X)$ is defined
$R'(\,\cdot\,|z,a)$-almost surely  uniquely in $y\in\Y$; Bertsekas
and Shreve \cite[Corollary~7.27.1]{BS} or Dynkin and Yushkevich
\cite[Appendix 4.4]{DY}. 
For a posterior distribution $z_t\in \P(\X)$, action $a_t\in \A$,
and an observation $y_{t+1}\in\Y,$ the posterior distribution
$z_{t+1}\in\P(\X)$ is
\begin{equation}\label{3.1}
z_{t+1}=H(z_t,a_t,y_{t+1}).
\end{equation}
However, the observation $y_{t+1}$ is not available in the COMDP
model, and therefore $y_{t+1}$ is a random variable with the
distribution $R'(\,\cdot\,|z_t,a_t)$, and the right-hand side of
(\ref{3.1}) maps $(z_t,a_t)\in \P(\X)\times\A$ to $\P(\P(\X)).$
Thus, $z_{t+1}$ is a random variable with values in  $\P(\X)$
whose distribution is defined uniquely by the stochastic kernel
\begin{equation}\label{3.7}
q(D|z,a):=\int_{\Y}\h\{H(z,a,y)\in D\}R'(dy|z,a),\quad D\in
\mathcal{B}(\P(\X)),\ z\in \P(\X),\ a\in\A;
\end{equation}
Hern\'andez-Lerma~\cite[p. 87]{HL}. The  particular choice of a
stochastic kernel $H$ satisfying (\ref{3.4}) does not effect the
definition of $q$ from (\ref{3.7}), since for each pair $(z,a)\in
\P(\X)\times\A$, the mapping $H(z,a,\cdot):\Y\to\P(\X)$ is defined
$R'(\,\cdot\,|z,a)$-almost surely uniquely in $y\in\Y$. 

The COMDP is defined as an MDP with the parameters
($\P(\X)$,$\A$,$q$,$\c$), where
(i) $\P(\X)$ is the state space; 
(ii) $\A$ is the
action set available at all states $z\in\P(\X)$; 
(iii) the
 one-step cost function $\c:\P(\X)\times\A\to{\bar \R}^1$, defined
\begin{equation}\label{eq:c}
\c(z,a):=\int_{\X}c(x,a)z(dx), \quad z\in\P(\X),\, a\in\A;
\end{equation}
(iv) transition probabilities $q$ on $\P(\X)$
given $\P(\X)\times \A$ defined in (\ref{3.7}).

For an MDP, a nonrandomized policy is called \textit{Markov}, if
all decisions depend only on the current state and time. A Markov
policy is called \textit{stationary}, if all decisions depend only
on current states.

For MDPs, Feinberg et al.~\cite[Theorem 2]{FKN} 
provides general conditions for the existence of optimal policies,
validity of optimality equations, and convergence of value
iterations. Here we formulate these conditions for an MDP whose
action sets in all states are equal, and then
Theorem~\ref{teor4.3} adapts Feinberg et al.~\cite[Theorem 2]{FKN}
to POMDPs.

\noindent\textbf{Assumption (${\rm \bf W^*}$)} (cf. Feinberg et al.~\cite{FKZ} and Lemma 2.5 in \cite{FKN}). 
(i) the function $c$ is $\K$-inf-compact on $\X\times\A$;
(ii) the transition probability  $P(\,\cdot\,|x,a)$ is weakly
continuous in $(x,a)\in \X\times\A$.

For the COMDP, Assumption \textbf{(${\rm \bf W^*}$)}
has the following form:
(i)  $\c$ is $\K$-inf-compact on $\P(\X)\times\A$;
 (ii) the transition probability  $q(\,\cdot\,|z,a)$ is weakly
continuous in $(z,a)\in \P(\X)\times\A$.

In the following theorem, the notation $\bar v$ is used for the
expected total costs for COMDPs instead the symbol $V$ used for
POMDPs. The following theorem follows directly from Feinberg et
al. \cite[Theorem~2]{FKNMOR} applied to the COMDP
$(\P(\X),\A,q,\c)$.

\begin{theorem}{\rm (Feinberg et al. \cite[Theorem~3.1]{FKZ}).}
 \label{teor4.3} Let either Assumption
{\rm{\bf({\bf D})}} or Assumption {\rm\bf({\bf P})} hold. If the
COMDP $(\P(\X),\A,q,\c)$ satisfies {\rm Assumption \textbf{(${\rm
\bf W^*}$)}},    then:

{(i}) the functions ${\bar v}_{t,\alpha}$, $t=0,1,\ldots$, and
${\bar v}_\alpha$ are lower semi-continuous on $\P(\X)$, and
${\bar v}_{t,\alpha}(z)\to 
 {\bar v}_\alpha (z)$ as $t \to \infty$ for all
$z\in \P(\X);$

{(ii)} for each $z\in \P(\mathbb{X})$ and $t=0,1,...,$
\begin{equation}\label{eq433}
\begin{aligned}
&\qquad\qquad{\bar v}_{t+1,\alpha}(z)=\min\limits_{a\in
\A}\left\{\c(z,a)+\alpha \int_{\P(\X)} {\bar
v}_{t,\alpha}(z')q(dz'|z,a)\right\}=
\\ &\min\limits_{a\in \A}\left\{\int_{\X}c(x,a)z(dx) +\alpha \int_{\X}\int_{\X}\int_{\Y} {\bar v}_{t,\alpha}(H(z,a,y)) Q(dy|a,x')P(dx'|x,a)z(dx) \right\},
\end{aligned}
\end{equation}
where ${\bar v}_{0,\alpha}(z)=0$ for all $z\in \P(\X)$, and the
nonempty sets
\[
A_{t,\alpha}(z):=\left\{a\in \A:\,{\bar
v}_{t+1,\alpha}(z)=\c(z,a)+\alpha \int_{\P(\X)} {\bar
v}_{t,\alpha}(z')q(dz'|z,a) \right\},\quad z\in \P(\X),\
t=0,1,\ldots,
\]
satisfy the following properties: (a) the graph ${\rm
Gr}(A_{t,\alpha})=\{(z,a):\, z\in\P(\X), a\in A_{t,\alpha}(z)\}$,
$t=0,1,\ldots,$ is a Borel subset of $\P(\X)\times \mathbb{A}$,
and (b) if ${\bar v}_{t+1,\alpha}(z)=+\infty$, then
$A_{t,\alpha}(z)=\A$ and, if ${\bar v}_{t+1,\alpha}(z)<+\infty$,
then $A_{t,\alpha}(z)$ is compact;

{(iii)} for each $T=1,2,\ldots$, for the COMDP there exists an
optimal Markov $T$-horizon policy $(\phi_0,\ldots,\phi_{T-1})$,
and if for a $T$-horizon Markov policy
$(\phi_0,\ldots,\phi_{T-1})$ the inclusions $\phi_{T-1-t}(z)\in
A_{t,\alpha}(z)$, $z\in\P(\X),$ $t=0,\ldots,T-1,$ hold, then this
policy is $T$-horizon optimal;

{(iv)} for  each $z\in
\P(\X)$ 
\begin{equation}\label{eq5a}
\begin{aligned}
 &\qquad\qquad{\bar
v}_{\alpha}(z)=\min\limits_{a\in
\A}\left\{\c(z,a)+\alpha\int_{\P(\X)} {\bar
v}_{\alpha}(z')q(dz'|z,a)\right\}=\\
&\min\limits_{a\in \A}\left\{\int_{\X}c(x,a)z(dx) +\alpha
\int_{\X}\int_{\X}\int_{\Y} {\bar v}_{\alpha}(H(z,a,y))
Q(dy|a,x')P(dx'|x,a)z(dx) \right\},\
\end{aligned}
\end{equation}
and the nonempty sets
\[
A_{\alpha}(z):=\left\{a\in \A:\,{\bar
v}_{\alpha}(z)=\c(z,a)+\alpha\int_{\P(\X)} {\bar
v}_{\alpha}(z')q(dz'|z,a) \right\},\quad z\in \P(\X),
\]
satisfy the following properties: (a) the graph ${\rm
Gr}(A_{\alpha})=\{(z,a):\, z\in\P(\X), a\in \A_\alpha(z)\}$ is a
Borel subset of $\P(\X)\times \mathbb{A}$, and (b) if ${\bar
v}_{\alpha}(z)=+\infty$, then $A_{\alpha}(z)=\A$ and, if ${\bar
v}_{\alpha}(z)<+\infty$, then $A_{\alpha}(z)$ is compact.

{(v)} for an infinite horizon problem there exists a stationary
discount-optimal policy $\phi_\alpha$ for the COMDP, and a
stationary policy $\phi_\alpha^{*}$ for the COMDP is optimal if
and only if $\phi_\alpha^{*}(z)\in A_\alpha(z)$ for all $z\in
\P(\X).$

{(vi)}  if $\c$ is inf-compact on $\P(\X)\times\A$, then the
functions ${\bar v}_{t,\alpha}$, $t=1,2,\ldots$, and ${\bar
v}_\alpha$ are inf-compact on $\P(\X)$.
\end{theorem}

Theorem~\ref{teor4.3} establishes the existence of stationary
optimal policies, validity of optimality equations, and
convergence of value iterations to optimal values under the
following natural conditions: (i) Assumption ({\bf D}) or ({\bf
P}) and the function $\bar c$ is $K$-inf-compact, and (ii) the
stochastic kernel $q$ on $\P(\X)$ given $\P(\X)\times A$ is weakly
continuous. 
Theorems~\ref{th:wstar} and \ref{t:totalvar} provide sufficient
conditions for (i) and (ii) respectively in terms of the
properties of the cost function $c$ and stochastic  kernels $P$ and $Q$. 

%
\begin{theorem}\label{th:wstar} {\rm (Feinberg et al. \cite[Theorem~3.4]{FKZ}).}
If the stochastic kernel  $P(dx'|x,a)$ on $\X$ given $\X\times\A$
is weakly continuous and the cost function $c:\X\times\A\to {\bar \R}^1$ is
bounded below and $\K$-inf-compact  on $\X\times\A$, then the cost
function $\c:\P(\X)\times\A\to{\bar \R}^1$ defined for the COMDP in
(\ref{eq:c}) is bounded from below by the same constant as $c$ and
$\K$-inf-compact on $\P(\X)\times\A$.
\end{theorem}

\begin{theorem}\label{t:totalvar} {\rm (Feinberg et al. \cite[Theorem~3.7]{FKZ}).}
The weak continuity of the stochastic kernel $P(dx'|x,a)$ on $\X$
given $\X\times\A$ and continuity in the total variation of the
stochastic kernel $Q(dy|a,x)$ on $\Y$ given $\A\times\X$ imply
that  the stochastic kernel $q(dz'|z,a)$ on $\P(\X)$ given
$\P(\X)\times\A$ is weakly continuous.
\end{theorem}

The following assumption, that has similarities with (\ref{result}), and theorem are used in Feinberg et al. \cite{FKZ} to prove
Theorem~\ref{t:totalvar}.


\noindent\textbf{Assumption {\bf(H)}}. There exists a stochastic
kernel $H$ on $\X$ given $\P(\X)\times\A\times\Y$ satisfying
(\ref{3.4}) such that: if a sequence
$\{z^{(n)}\}_{n=1,2,\ldots}\subseteq\P(\X)$ converges weakly to
$z\in\P(\X)$, and a sequence $\{a^{(n)}\}_{n=1,2,\ldots}\subseteq\A$
converges to $a\in\A$ as $n\to\infty$, then there exists a
subsequence $\{(z^{(n_k)},a^{(n_k)})\}_{k=1,2,\ldots}\subseteq
\{(z^{(n)},a^{(n)})\}_{n=1,2,\ldots}$ and a measurable subset $C$ of
$\Y$ such that $R'(C|z,a)=1$ and for all $y\in C$
\begin{equation}\label{eq:ASSNH}
H(z^{(n_k)},a^{(n_k)},y)\mbox{ converges weakly to }H(z,a,y). 
\end{equation}
In other words, \eqref{eq:ASSNH} holds $R'(\,\cdot\,|z,a)$-almost
surely.

According to the following theorem, if the stochastic kernel $R'$ is setwise continuous and Assumption~{\bf(H)} holds, then the stochastic  kernel $q$ is weakly continuous.  According to Feinberg et al.~\cite[Theorem 3.7]{FKZ}, weak continuity of the stochastic  kernel $P$ and continuity of the observation stochastic kernel  $Q$ in the total variation imply that the stochastic kernel  $R'$ is setwise continuous and Assumption~{\bf(H)} holds.  Another sufficient condition for weak continuity of $q$ is that there is a weakly continuous version of a stochastic kernel $H$ on $\X$ given $\P(\X)\times\A\times\Y$; see  Striebel~\cite{St} and Hern\'andez-Lerma~\cite{HL}.  However, this condition may not hold for a POMDP with a weakly continuous stochastic  kernel $P$ and a observation stochastic kernel  $Q$ continuous in the total observation; see Feinberg et al.~\cite[Example 4.2]{FKZ}.

\begin{theorem}\label{th:contqqq2} {\rm (Feinberg et al. \cite[Theorem~3.5]{FKZ}).} If
the stochastic kernel $R'(dy|z,a)$ on $\Y$ given
$\P(\X)\times\A$ is setwise continuous and  Assumption~{\bf(H)}
holds, then the stochastic kernel $q(dz'|z,a)$ on $\P(\X)$ given
$\P(\X)\times\A$ is weakly continuous.
\end{theorem}

In addition to Theorem~\ref{t:totalvar}, that provides the sufficient condition of weak continuity of a stochastic  kernel $q$ in terms of transition and observation probabilities $P$ and $Q,$ and to  Theorem~\ref{th:contqqq2}, that provides the sufficient condition of weak continuity of a stochastic  kernel $q$ in terms of stochastic kernels $R'$ and $H,$ a sufficient condition can be formulated in terms of the stochastic kernel $R$  on $\X\times\Y$ given
${\P}(\X)\times \A$, defined in (\ref{3.3}). For each $B\in\tau(\X)$  consider the family of functions
\[\mathcal{R}_B=\{  {\P}(\X)\times \A\to R(B\times C|z,a):\, C\in \tau(\Y)\}\] mapping ${\P}(\X)\times \A$ into $[0,1]$.
\begin{theorem}\label{teor:Rtotvar}
Let the topology on $\X$ have a countable base
$\tau_b^\X$ with the following two properties:
\begin{itemize} \item[(a)]$\X\in\tau_b^\X$, \item[(b)] for
each finite intersection $\oo=\cap_{i=1}^ k {\oo}_{i}$ of sets
$\oo_{i}\in\tau_b^\X,$ $i=1,2,\ldots,k,$
the family of
functions $\mathcal{R}_\oo$ is equicontinuous at all the points
$(z,a)\in \P(\X)\times\A$. \end{itemize} Then the following two statements take place:
\begin{itemize}
 \item[(i)] the stochastic kernel $R'(dy|z,a)$ on $\Y$ given
$\P(\X)\times\A$ is  continuous in the total variation, and  Assumption~{\bf(H)}
holds;
 \item[(ii)] the stochastic kernel 
$q(dz'|z,a)$ on $\P(\X)$ given
$\P(\X)\times\A$ is weakly continuous.\end{itemize}
\end{theorem}
\begin{proof}
(i) The equicontinuity at all the points $(z,a) \in \P(\X)
\times \A$ of the family of functions
$\mathcal{R}_{\oo}$ defined on $\P(\X) \times \A$,  being applied to $\oo = \X,$ implies that the stochastic kernel $R'$ on $\X$ given
$\P(\X)\times\A$ is continuous in the total variation.   Theorem~\ref{mainthkern}, being applied to the Borel subsets of Polish spaces $\S_1=\X,$ $\S_2=\Y,$  and $\S_3=\P(\X)\times\A,$ yields that
Assumption~({\bf H}) holds.  (ii) Since the continuity of $R'$ in the total variations implies its setwise continuity, the  statement follows from statement (i) and Theorem~\ref{th:contqqq2}.
\end{proof}

The following theorem completes the descriptions of the relations between the assumptions of Theorems~\ref{t:totalvar}--\ref{teor:Rtotvar}.  Among these three groups of assumptions, the assumptions of Theorem~\ref{th:contqqq2} are the most general, and they follow from the assumptions of Theorem~\ref{teor:Rtotvar}, which in its turn follow from the assumptions of Theorem~\ref{t:totalvar}.

\begin{theorem}\label{t:tvimplr} If
 the stochastic kernel $P(dx'|x,a)$ on $\X$
given $\X\times\A$ is weakly continuous and
the stochastic kernel $Q(dy|a,x)$ on $\Y$ given $\A\times\X$  is continuous in the total variation, then the assumptions of
Theorem~\ref{teor:Rtotvar} hold.
\end{theorem}
\begin{proof}
In view of Feinberg et al.~\cite[Lemma 5.3]{FKZ}, the family of function $\mathcal{R}_{\oo_1\setminus\oo_2}$ is equicontinuous for two arbitrary open subsets $\oo_1$ and $\oo_2$ in $\X.$  By setting $\oo_2=\emptyset,$  this result implies that the family of functions   $\mathcal{R}_{\oo}$ is equicontinuous for each open subset $\oo$ in $\X.$  Since we endowed $\X$ with the induced
topology from a separable metric space, its topology has a countable base which is closed according to the finite intersections.
Therefore, this countable base of the topology on $\X$  satisfies   assumptions of
Theorem~\ref{teor:Rtotvar}.
\end{proof}


%

Observe that Theorem~\ref{t:totalvar} follows from Theorems~\ref{teor:Rtotvar} and~\ref{t:tvimplr}. The following theorem provides sufficient conditions for the
existence of optimal policies for the COMDP.  Its first statement is Theorem~\ref{t:totalvar}, which is
repeated for completeness of the statements.


\begin{theorem}\label{main} {\rm (Feinberg et al. \cite[Theorem~3.6]{FKZ}).}
Let either Assumption {\rm\bf({\bf D})} or Assumption
{\rm\bf({\bf P})} hold. If the function $c$ is $\K$-inf-compact on
$\X\times\A$ then each of the following conditions:
\begin{itemize}
\item[(i)]  the stochastic kernel $P(dx'|x,a)$ on $\X$ given
$\X\times\A$ is weakly continuous, and  the stochastic kernel
$Q(dy|a,x)$ on $\Y$ given $\A\times\X$ is continuous in the total
variation;
\item[(ii)] the assumptions of Theorem~\ref{teor:Rtotvar} hold;
\item[(iii)] the stochastic kernel $R'(dy|z,a)$ on $\Y$ given
$\P(\X)\times\A$ is setwise continuous and  Assumption~{\bf(H)}
holds,
\end{itemize}
implies that the COMDP $(\P(\X),\A,q,\c)$ satisfies {Assumption
{\rm\bf(${\rm \bf W^*}$)}}, and therefore statements (i)--(vi) of
Theorem~\ref{teor4.3} hold.
\end{theorem}
\begin{proof} Theorem~\ref{th:wstar} implies that the cost function $\bar{ c}$ for the COMDP is bounded below and
$\K$-inf-compact on $\P(\X)\times \A$.  Weak continuity of the stochastic kernel $q$ on $\P(\X)$ given $\P(\X)\times \A$
follows from Theorems~\ref{t:totalvar}--\ref{teor:Rtotvar}.
\end{proof}


Example~4.1 from Feinberg et al. \cite{FKZ} demonstrates that, if
the stochastic kernel $Q(dy|a,x)$ on $\Y$ given $\A\times\X$ is
setwise continuous, then the transition probability $q$ for the
COMDP may not be weakly continuous in $(z,a)\in\P(\X)\times\A$. In
that example the state set consists of two points. Therefore, if
the stochastic kernel $P(dx'|x,a)$ on $\X$ given $\X\times\A$ is
setwise continuous (even if it is continuous in the total
variation) in $(x,a)\in\X\times\A$ then the setwise continuity of
the stochastic kernel $Q(dy|a,x)$ on $\Y$ given $\A\times\X$  is
not sufficient for the weak continuity
of $q$.

\section{Markov Decision Models with Incomplete Information}\label{S5}
Consider a Markov decision model  with incomplete
information (MDMII); Dynkin and  Yushkevich~\cite[Chapter 8]{DY},
Rhenius~\cite{Rh}, Yushkevich~\cite{Yu} (see also Rieder \cite{Ri}
and B\"auerle and Rieder~\cite{BR} for a version of this model
with transition probabilities having densities). This model is
defined by an \textit{observed state space} $\Y$, an
\textit{unobserved state space} $\W$, an \textit{action space}
$\A$, nonempty \textit{sets of available actions} $A(y),$ where
$y\in\Y$,  a stochastic kernel $P$ on $\Y\times\W$ given
$\Y\times\W\times\A$, and a one-step cost function $c:\, G\to  {\bar \R}^1,$
where $G=\{(y,w,a)\in \Y\times\W\times\A:\, a\in A(y)\}$ is the
graph of the mapping $A(y,w)=A(y),$ $(y,w)\in \Y\times\W.$ Assume
that:

(i) $\Y$, $\W$ and $\A$ are Borel subsets of Polish spaces. For
all $y\in \Y$ a nonempty Borel subset $A(y)$ of $\mathbb{A}$
represents the \textit{set of actions} available at $y;$

(ii) the graph  of the mapping $A:\Y\to 2^\A$, defined as $ {\rm
Gr} ({A})=\{(y,a) \, : \,  y\in \Y, a\in A(y)\}$ is measurable,
that is, ${\rm Gr}(A)\in {\mathcal B}(\Y\times\A)$, and this graph
allows a measurable selection, that is,  there exists a measurable
mapping $\phi:\Y\to \mathbb{A}$ such that $\phi(y)\in A(y)$ for
all $y\in \Y$;

(iii) the stochastic  kernel $P$ on $\X$ given $\Y\times\W\times\A$
is weakly continuous in $(y,w,a)\in \Y\times\W\times\A$;

(iv) the one-step cost function $c$ is $\K$-inf-compact on $G$, that is,
for each compact set $K\subseteq\Y\times\W$ and for each
$\lambda\in \mathbb{R}^1$, the set ${\cal
D}_{K,c}(\lambda)=\{(y,w,a)\in G:\, c(y,w,a)\le\lambda\}$ is
compact.

Let us define $\X=\Y\times\W,$ and for $x=(y,w)\in \X$ let us
define $Q(C|x)= {\bf I}\{y\in C\}$ for all $C\in {\cal B}(\Y).$
Observe that this $Q$ corresponds to the  continuous function $y=
F(x),$ where $F(y,w)=y$ for all $x=(y,w)\in\X$ (here $F$ is a
projection of $\X=\Y\times\W$ on $\Y$).  Thus, as explained in
Example~4.1 from Feinberg et al. \cite{FKZ}, the stochastic kernel $Q(dy|x)$ is weakly
continuous in $x\in\X.$ Then by definition, an MDMII is a POMDP with
the state space $\X$, observation set $\Y$, action space $\A$,
available action sets $A(y)$, stochastic  kernel $P$, observation
kernel $Q(dy|a,x):=Q(dy|x)$, and one-step cost function $c$.
However, this model differs from our basic definition of a POMDP
because action sets $A(y)$ depend on observations and one-step
costs $c(x,a)=c(y,w,a)$ are not defined when $a\notin A(y).$ To
avoid this difficulty, we set $c(y,w,a)=+\infty$ when $a\notin
A(y)$.  The extended function $c$ is $\K$-inf-compact on
$\X\times\A$ because the set ${\cal D}_{K,c}(\lambda)$ remains
unchanged for each $K\subseteq\Y\times\W$ and for each
$\lambda\in\mathbb{R}^1.$

Thus, an MDMII is a special case of a POMDP $(\X,\Y, \A,P,Q,c)$,
when $\X=\Y\times\W$ and the observation kernel $Q$ is 
defined by the projection of $\X$ on $\Y.$ The observation stochastic kernel
$Q(\,\cdot\,|x)$ is weakly continuous in $x\in \X$.  This is
weaker that the continuity of $Q$ in the total variation that,
according to Theorem~\ref{main},  ensures weak continuity of the
stochastic kernel for the COMDP and the existence of optimal
policie. Indeed, Feinberg et al. \cite[Example 8.1]{FKZ}
demonstrates that even under the stronger assumption, that $P$
is setwise continuous, the corresponding stochastic  kernel $q$  on $\P(\X)$ given $\P(X)\times\A$ may not be weakly
continuous.

The natural question is: which conditions are sufficient for the
existence of optimal policies for the MDMII?  Since an MDMII is a particular POMDP, the existence of optimal policies for an MDMII is equivalent to the existence of optimal policies for the COMDP corresponding to this
MDMII.  Theorem~\ref{teor4.3} gives an answer in a general form by stating that such conditions are  the week continuity of the transition probability $q$ of the corresponding COMDP and the $\K$-inf-compactness of the cost function $\bar c$ for the COMDP.  The following theorem provides a sufficient condition for the weak continuity of $q$.
  For each
open set $\oo$ in $\W$ consider the
family of functions
$\mathcal{P}^*_\oo=\{  (x,a)\to P(C\times\oo|x,a):\, C\in
\tau(\Y)\}$ mapping $\X\times\A$ into $[0,1]$.

\begin{theorem}\label{t:totalvar1}
Let the topology on $\W$ have a countable base
$\tau_b^\W$ satisfying the following two conditions:
\begin{itemize}
 \item[(i)] $\W\in\tau_b^\W,$  
 \item[(ii)]  for
each finite intersection $\oo=\cap_{i=1}^ k {\oo}_{i}$ of sets
$\oo_{i}\in\tau_b^\W,$ $i=1,2,\ldots,k,$
the family of
functions $\mathcal{P}^*_\oo$ is equicontinuous at all the points
$(x,a)\in \X\times\A$.
\end{itemize} Then  the stochastic kernel 
$q(dz'|z,a)$ on $\P(\X)$ given
$\P(\X)\times\A$ is weakly continuous.
\end{theorem}
%
%

\begin{proof}
Let $\tau_b^\Y$ be a countable base of the
topology on $\Y$ closed with respect to the finite intersections. Such base exists, because $\Y$ is the separable metric space. Since finite intersections
of elements of the base $\tau_b^\W$ are open sets, let us choose
$\tau_b^\W$ in a way that finite intersections of elements of
$\tau_b^\W$ belong to $\tau_b^\W.$ Then
$\tau_b^\X:=\{\oo_\Y\times\oo_\W:\, \oo_\Y\in\tau_b^\Y,\, \oo_\W\in\tau_b^\W \}$
is the countable base of the topology on
$\X =\Y\times \W$ defined by the products of the topologies on
$\Y$ and $\W$ and for any finite tuples of open sets
$\{\oo_\Y^{(j)}\}_{j=1}^{N}$ in $\Y$ and
$\{\oo_\W^{(j)}\}_{j=1}^{N}$ in $\W$, $N=1,2,\ldots,$ their finite
intersections $\cap_{j=1}^N\oo_\Y^{(j)}$ and
$\cap_{j=1}^N\oo_\W^{(j)}$ are open in $\Y$ and $\W$
respectively. Moreover,   $\cap_{j=1}^N \left(\oo_\Y^{(j)}\times
\oo_\W^{(j)}\right)= \left(\cap_{j=1}^N\oo_\Y^{(j)}\right)\times
\left(\cap_{j=1}^N\oo_\W^{(j)}\right)\in\tau_b^\X$ for any finite
tuples of open sets $\{\oo_\Y^{(j)}\}_{j=1}^{N}$ from $\tau_b^\Y$ and
$\{\oo_\W^{(j)}\}_{j=1}^{N}$ from $\tau_b^\W$.
From (\ref{3.3}) it follows that
\[
R(C_1\times B\times C_2|z,a)=\int_{\X}P((C_1\cap C_2)\times
B|x,a)z(dx),\quad\ B\in \mathcal{B}(\W),\ C_1,C_2\in
\mathcal{B}(\Y),\ z\in\P(\X),\ a\in \A,
\]
\[
R'(C|z,a)=\int_{\X}P(C\times \W|x,a)z(dx),\qquad
 C\in \mathcal{B}(\Y),\ z\in\P(\X),\ a\in
\A.\] For any nonempty open sets $\oo_\Y\in \tau_b^\Y$ and $\oo_\W\in \tau_b^\W$
respectively, Theorem~\ref{kern}, with $\S_1 = \P(\X)$, $\S_2 = \X$,
$\S_3 = \A$, $\oo = \X$, $\Psi(B | z) = z(B)$, and $\mathcal{A}_0 =
\{(x,a) \to P((\oo_\Y\cap C) \times \oo_\W)|x,a): C \in \tau(\Y)\}$,
implies the equicontinuity of the family of functions
\[
\mathcal{R}_{\oo_\Y\times\oo_\W}=\left\{(z,a)\to R(\oo_\Y\times
\oo_\W\times C|z,a)\,:\, C\in\tau(\Y)\right\},
\]
defined on $\P(\X) \times \A$, at all the points $(z,a) \in \P(\X)
\times \A$. 
Therefore, Theorem~\ref{teor:Rtotvar}(ii) yields that 
the stochastic kernel $q(dz'|z,a)$ on $\P(\X)$ given
$\P(\X)\times\A$ is weakly continuous.
\end{proof}

Assumptions of Theorem~\ref{t:totalvar1} are weaker  than equicontinuity at all the points $(x,a)\in \X\times\A$ of the
family of functions $\mathcal{P}_\oo$ for all open sets $\oo$ in $\W$ (see Example~\ref{exa:MDM} above), which in its turn is a weaker assumption than
the continuity of the stochastic kernel $P$ on $\X$ given
$\X\times\A$ in the total variation. 
The following theorem
states sufficient conditions for the existence of optimal policies
for MDMIIs, the validity of optimality equations, and convergence of
value iterations to optimal values.  Theorem~\ref{teor:Ren} generalizes \cite[Theorem 8.2]{FKZ}, where the equicontinuity at all the points $(x,a)\in \X\times\A$ of the
family of functions $\mathcal{P}^*_\oo$ for all open sets $\oo$ in $\W$ is assumed.

\begin{theorem}\label{teor:Ren}
Let either Assumption~{\rm\bf(D)} or Assumption~{\rm\bf(P)} hold, and let the cost function $c$ be $\K$-inf-compact on $G$. 
 If the topology on $\W$ has a countable base
$\tau_b^\W$ satisfying assumptions (i) and (ii) of Theorem~\ref{t:totalvar1}, 
then the COMDP $(\P(\X),\A,q,\c)$ satisfies {\rm Assumption \textbf{(${\rm
\bf W^*}$)}}, and
therefore the conclusions of Theorem~\ref{teor4.3} hold.
\end{theorem}
\begin{proof}
Assumption  \textbf{(${\rm
\bf W^*}$)}(i) follows from Corollary~\ref{teor:2} and Theorem~\ref{th:wstar}. Assumption  \textbf{(${\rm
\bf W^*}$)}(ii) follows from Theorem~\ref{t:totalvar1}. Therefore, the COMDP $(\P(\X),\A,q,\c)$ satisfies {\rm Assumption \textbf{(${\rm
\bf W^*}$)}} and
the conclusions of Theorem~\ref{teor4.3} hold.
\end{proof}

 {\bf Acknowledgements.}   The authors thank   M. Mandava for providing  useful remarks. The research of the
first author was partially supported by NSF grant CMMI-1335296.

\end{document}